\documentclass[10pt]{amsart}

\usepackage[foot]{amsaddr}
\usepackage[latin1]{inputenc}

\usepackage{amsmath,amssymb,amsthm,amscd}
\usepackage{colordvi,graphicx}
\usepackage[margin=1in]{geometry}
\usepackage{hyperref}
\usepackage[enableskew]{youngtab}
\usepackage[all,cmtip]{xy}
\usepackage{color}
\usepackage{qtree}
\usepackage{amsfonts}
\usepackage{amsaddr}
\usepackage{enumitem}
\usepackage{tikz}
\usetikzlibrary{math,calc}
\usepackage{hyphenat}
\usepackage{booktabs}

\newtheorem{theorem}{Theorem}[subsection]
\newtheorem{lemma}[theorem]{Lemma}
\newtheorem{proposition}[theorem]{Proposition}
\newtheorem{corollary}[theorem]{Corollary}

\theoremstyle{definition}
\newtheorem{definition}[theorem]{Definition}
\newtheorem{example}[theorem]{Example}

\newtheorem{remark}[theorem]{Remark}

\numberwithin{equation}{subsection}
\numberwithin{table}{section}
\numberwithin{figure}{section}

\def\CC{{\mathbb C}}\def\FF{{\mathbb F}}
\def\qand{\quad\text{and}\quad}

\newcommand{\IN}{\ensuremath{\mathbb{N}}}

\newcommand{\NOR}{\ensuremath{\mathbin{\downarrow}}}

\newcommand{\nset}[1]{\ensuremath{[{#1}]}}
\newcommand{\card}[1]{\ensuremath{\lvert{#1}\rvert}}
\newcommand{\varcount}[1]{\ensuremath{\lvert{#1}\rvert}}

\newcommand{\spac}[2][n]{\ensuremath{s^\mathrm{ac}_{#1}({#2})}}
\newcommand{\finespac}[2][n]{\ensuremath{\sigma^\mathrm{ac}_{#1}({#2})}}
\newcommand{\spa}[2][n]{\ensuremath{s^\mathrm{a}_{#1}({#2})}}
\newcommand{\finespa}[2][n]{\ensuremath{\sigma^\mathrm{a}_{#1}({#2})}}

\renewcommand{\SS}{{\mathfrak S}}

\def\RR{\mathbb{R}}
\def\P{\mathcal P}

\DeclareMathOperator{\pr}{pr}
\DeclareMathOperator{\var}{var}

\title{The associative-commutative spectrum of a binary operation}

\author{Jia Huang}
\address[J. Huang]{Department of Mathematics and Statistics, University of Nebraska at Kearney, Kearney, NE 68849, USA}
\email{huangj2@unk.edu}

\author{Erkko Lehtonen}

\address[E. Lehtonen]%
     {Department of Mathematics \\
     Khalifa University \\
     P.O.~Box 127788 \\
     Abu Dhabi \\
     United Arab Emirates
     \and
     Centro de Matem\'atica e Aplica\c{c}\~oes \\
     Faculdade de Ci\^encias e Tecnologia \\
     Universidade Nova de Lisboa \\
     Quinta da Torre \\
     2829-516 Caparica \\
     Portugal}
\email{erkko.lehtonen@ku.ac.ae}

\thanks{This work is funded by National Funds through the FCT -- Funda\c{c}\~{a}o para a Ci\^{e}ncia e a Tecnologia, I.P., under the scope of the project UIDB/00297/2020 (Center for Mathematics and Applications) and the project PTDC/MAT-PUR/31174/2017.}


\begin{document}
\begin{abstract}
We initiate the study of a quantitative measure for the failure of a binary operation to be commutative and associative. 
We call this measure the associative\hyp{}commutative spectrum as it extends the associative spectrum (also known as the subassociativity type), which measures the nonassociativity of a binary operation.
In fact, the associative\hyp{}commutative spectrum (resp.\ associative spectrum) is the cardinality of the operad with (resp.\ without) permutations obtained naturally from a groupoid (a set with a binary operation).
In this paper we provide some general results on the associative\hyp{}commutative spectrum, precisely determine this measure for certain binary operations, and propose some problems for future study.
\end{abstract}

\maketitle

\section{Introduction}\label{sec:Intro}

Associativity and commutativity are important properties for binary operations.
Although many familiar operations satisfy both properties, some only satisfy one or neither of them.
Moreover, nonassociativity and noncommutativity arise in many algebraic structures, such as Lie algebras, Poisson algebras, and so on.
One can measure the failure of a binary operation to be associative by its associative spectrum, which we recall below. 

A \emph{groupoid} is a set $G$ with a single binary operation $*$.\footnote{Note that the term \emph{groupoid} has a different meaning in category theory.}
A \emph{bracketing} of $n$ variables is a groupoid term over the set $X_n := \{x_1, \dots, x_n\}$ of variables that is obtained by inserting parentheses in the word $x_1 x_2 \dots x_n$ in a valid way.
Let $\P_*(n)$ denote the set of all $n$\hyp{}ary term operations on $(G,*)$ induced by the bracketings of $n$ variables.
The cardinality $\card{\P_*(n)}$ measures to some extent the failure of $*$ to be associative.
In general, we have $1 \leq \card{\P_*(n)} \leq C_{n-1}$ where $C_n := \frac{1}{n+1}\binom{2n}{n}$ is the ubiquitous \emph{Catalan number}.

Cs\'{a}k\'{a}ny and Waldhauser~\cite{AssociativeSpectra1} called the sequence $(\spa{*})_{n \in \IN_{+}}$, where $\spa{*} := \card{\P_*(n)}$ and $\IN_{+}:=\{1,2,3,\ldots\}$, the \emph{associative spectrum} of the binary operation $*$, while Braitt and Silberger~\cite{Subassociative} named it the \emph{subassociativity type} of the groupoid $(G,*)$.
Independently, Hein and the first author~\cite{CatMod} also proposed the study of $\spa{*} = \card{\P_*(n)}$ for a binary operation $*$, and provided an explicit formula when $*$ satisfies \emph{$k$\hyp{}associativity} (a generalization of associativity).
The associative spectra of many other binary operations have been determined~\cite{VarCat, DoubleMinus, GraphAlgebra1, GraphAlgebra2,AssociativeSpectra2}.

For each $n\ge1$, let $\overline{\P}_*(n)$ be the set of all $n$-ary term operations induced on $(G,*)$ by \emph{full linear terms} of $n$ variables, i.e., groupoid terms over $X_n$ in which each variable $x_1, \dots, x_n$ occurs exactly once (but in an arbitrary order, as opposed to bracketings).
We call the sequence $(\spac{*})_{n \in \IN_{+}}$, where $\spac{*} := \card{\overline{\P}_*(n)}$, the \emph{associative\hyp{}commutative spectrum} (in brief, \emph{ac\hyp{}spectrum}) of the binary operation $*$, which measures both the nonassociativity and the noncommutativity of $*$.
We will determine the ac\hyp{}spectra for certain binary operations and exhibit some connections to other interesting combinatorial objects and results.

It turns out that the associative spectrum and the ac-spectrum have connections with the operad theory, which models both nonassociativity and noncommutativity by using binary trees.
It was developed by Boardman, May, Vogt, and others, with applications recently found in many branches of mathematics (see, e.g., Loday and Vallette~\cite{Operads}).
We recall some basic definitions below.

An \emph{operad without permutations} is an indexed family $\P = \{\P(n)\}_{n\ge1}$ of sets with an identity element $1\in \P(1)$ and, for all positive integers $n,m_1,\ldots,m_n$, a \emph{composition map} 
\begin{align}
{\circ} \colon \P(n)\times \P(m_1)\times\cdots\times \P(m_n) & \to \P(m_1+\cdots+m_n) \\
(P,P_1,\ldots,P_n) & \mapsto P \circ (P_1,\ldots,P_n)
\end{align}
satisfying the following conditions:
for any $P \in \P(n)$, $P_i \in \P(n_i)$, $P_{i,j} \in \P(n_{i,j})$ ($1 \leq i \leq n$, $1 \leq j \leq n_i$), we have
$P\circ(1,\ldots,1)=P = 1\circ P$ and
\begin{multline*}
P\circ(P_1\circ(P_{1,1},\ldots,P_{1,m_1}), \ldots, P_n\circ(P_{n,1},\ldots,P_{n,m_n})) \\
= (P\circ(P_1,\ldots,P_{n})) \circ(P_{1,1},\ldots,P_{1,m_1},\ldots,P_{n,1},\ldots,P_{n,m_n}).
\end{multline*}
(An operad without permutations can thus be seen as a many\hyp{}sorted algebra with a nullary operation $1$ and operations $\circ_{n,m_1,\dots,m_n}$ for all positive integers $n, m_1, \dots, m_n$, but for notational simplicity, the same symbol $\circ$ is used to denote all of the latter.)
The elements of $\P(n)$ are called \emph{$n$-ary operations}.\footnote{One can obtain an operad without permutations by taking $\P(n)$ to be the set of all $n$-ary operations on a set, or, as another example, all $n$-ary term operations of an algebraic structure. This motivates the terminology ``$n$-ary operation'' in the context of operads.}
The \emph{Hilbert series} of an operad $\P$ without permutations is $\sum_{n=1}^\infty \card{\P(n)} t^n$.

We call $\P = \{ \P(n) \}_{n\ge1}$ an \emph{operad with permutations} if $\P$ satisfies all of the above and has an action of the symmetric group $\SS_n$ on $\P(n)$ for each $n\ge1$ satisfying the following equivariance conditions: for any $P\in\P(n)$, $w\in \SS_n$, $P_i\in\P(m_i)$, and $w_i\in\SS_{m_i}$, 
\[ 
(P\cdot w)\circ(P_{w^{-1}(1)},\ldots,P_{w^{-1}(n)}) = (P\circ(P_1,\ldots,P_n))\cdot w, 
\]
\[
P\circ(P_1\cdot w_1,\ldots,P_n\cdot w_n) = (P\circ(P_1,\ldots,P_n))\cdot(w_1,\ldots,w_n).
\]
Here by abuse of notation, the permutation $w$ on the right side of the first equation is the permutation of the set $\{1, \dots, m_1 + \dots + m_n\}$ that breaks the set into $n$ consecutive blocks of sizes $m_1, \dots, m_n$ and then permutes the $n$ blocks by $w$.
The \emph{Hilbert series} of an operad $\P = \{ \P(n) \}_{n\ge1}$ with permutations is $\sum_{n=1}^\infty \frac{\card{\P(n)}}{n!} t^n$.

Given a groupoid $(G,*)$, we have an operad $\P_* := \{\P_*(n)\}_{n\ge1}$ without permutations and an operad $\overline{\P}_* := \{ \overline{\P}_*(n)\}_{n\ge1}$ with permutations, whose Hilbert series have coefficients given by the associative spectrum $\spa{*}$ and the ac-spectrum $\spac{*}$, respectively.

It is clear that $\spac{*} \ge1$; the equality holds for all $n \in \IN_{+}$ if and only if $*$ is both commutative and associative.
On the other hand, we have the following upper bounds for the ac-spectrum.
\begin{enumerate}[label={\upshape(\roman*)}]
\item
Since full linear terms over $X_n$ are in bijection with (ordered) binary trees with $n$ labeled leaves, we have $\spac{*}\le n!C_{n-1}$ for an arbitrary binary operation $*$.
\item
Since the equivalence classes of full linear terms over $X_n$ modulo the equational theory of associative groupoids (semigroups) are in bijection with permutations of $\{1, \dots, n\}$, we have $\spac{*} \leq n!$ if $*$ is an associative binary operation.
\item
Since the equivalence classes of full linear terms over $X_n$ modulo the equational theory of commutative groupoids are in bijection with unordered binary trees with $n$ labeled leaves, we have $\spac{*} \leq D_{n-1}$ 
if $*$ is commutative, where $D_n:=(2n)!/(2^{n}n!)$~\cite[A001147]{OEIS}.
\end{enumerate} 
The upper bound in the last case is the solution to Schr\"oder's third problem; see, e.g., Stanley~\cite[p.~178]{EC2}.

In Section~\ref{sec:Free} we show that the above upper bounds can be achieved by the free groupoid on one generator, the free associative groupoid (i.e., the free semigroup) on two generators, and the free commutative groupoid on one generator, respectively.

In Section~\ref{sec:associative-commutative} we focus on binary operations that are associative or commutative.
For an associative noncommutative binary operation $*$, we show that its ac-spectrum $\spac{*}$ attains the upper bound $n!$ if it has a neutral element (i.e., identity element), and give some other examples for which $\spac{*}<n!$ in Section~\ref{sec:associative}.
In Section~\ref{sec:commutative} and Section~\ref{sec:commutative-neutral}, we provide some concrete examples of commutative groupoids whose ac\hyp{}spectra reach the upper bound $D_{n-1}$, but for the arithmetic, geometric, and harmonic means, we show that their ac\hyp{}spectra coincide with an interesting sequence that counts ways to express $1$ as an ordered sum of powers of $2$~\cite[A007178]{OEIS}.
The last example shows that the ac\hyp{}spectrum of a commutative operation may not achieve the upper bound $D_{n-1}$ even if it is \emph{totally nonassociative}, i.e., its associative spectrum equals the upper bound $C_{n-1}$.
However, we show that the converse does hold: a commutative groupoid is totally nonassociative if its ac\hyp{}spectrum reaches that upper bound.

In Section~\ref{sec:anticommutative} we show that the ac-spectrum of the bilinear product of some anticommutative algebras over a field, including the cross product and certain Lie brackets, is exactly two times the upper bound $D_{n-1}$ for the ac-spectrum of a commutative operation.

In Section~\ref{sec:TN}, we determine the ac\hyp{}spectra of some more examples of totally nonassociative operations, including the exponentiation, the implication, and the negated disjunction (NOR).
The exponentiation and the converse implication satisfy the identity $x(yz) \approx x(zy)$ and their ac-spectrum reaches the upper bound $n^{n-1}$ for the ac-spectrum of any binary operation satisfying the above identity. 
Here $n^{n-1}$ shows up because it is the number of unordered rooted trees with $n$ labeled vertices.
The negated disjunction is commutative and its ac-spectrum reaches the upper bound $D_{n-1}$ for commutative operations.
Together with Example~\ref{ex:two-element-groupoids}, this completely describes the ac-spectra of all two\hyp{}element groupoids.

In Section~\ref{sec:DepthEq}, we see that for some groupoids, two full linear terms induce the same term operation if and only if the corresponding binary trees are equivalent with respect to certain attributes related to the depths of the leaves.
In Section~\ref{sec:RightDepth} we obtain a formula involving the Stirling numbers of the second kind for the ac\hyp{}spectrum of a binary operation $*$  satisfying the property that any two full linear terms agree on $*$ if and only if the right depth sequences of their corresponding binary trees are congruent modulo $k$ (this is also equivalent to the the $k$-associativity mentioned earlier).
An example is given by $a*b:=a+e^{2\pi i/k} b$, which becomes addition and subtraction when $k=1$ and $k=2$.
Related to this example is the operation $a*b:= e^{2\pi i/k}(a+b)$.
When $k=2$ this becomes the \emph{double minus operation} $a \ominus b := -a-b$ whose associative spectrum is $\card{\P_*(n)} = \lfloor 2^n/3 \rfloor$~\cite[A000975]{OEIS}, as shown by Cs\'{a}k\'{a}ny and Waldhauser~\cite[\S~5.3]{AssociativeSpectra1} and independently by the first author, Mickey, and Xu~\cite{DoubleMinus}.
In section~\ref{sec:k-depth} we show that $\spac{\ominus}=(2^n-(-1)^n)/3$, which is the well-known \emph{Jacobsthal sequence}~\cite[A001045]{OEIS}.
The more general operation $a*b:= e^{2\pi i/k}(a+b)$ satisfies the \emph{$k$\hyp{}depth\hyp{}equivalence} studied recently by the second author.
Computations show that neither the associative spectrum nor the ac\hyp{}spectrum of this operation matches with any existing sequence in OEIS~\cite{OEIS} when $k>2$.

In Section~\ref{sec:conclusion}, we make some concluding remarks and indicate possible directions for further research.

For the reader's convenience, we provide a summary in Table~\ref{tab:summary}, in which we indicate for the different types of groupoids considered in this paper their ac-spectra and, for the sake of comparison, also their associative spectra.

\begin{table}[ht]
\renewcommand{\arraystretch}{1.5}
\begin{center}
\begin{tabular}{p{6.5cm}ccc}
\toprule
groupoid $(G,*)$ & $\spac{*}$ & $\spa{*}$ & reference \\
\midrule
free on one generator & $n!C_{n-1}$ & $C_{n-1}$ & Proposition~\ref{prop:free-groupoid} \\
free associative on two generators & $n!$ & $1$ & Proposition~\ref{prop:free-semigroup} \\
free commutative on one generator & $D_{n-1}$ & $C_{n-1}$ & Proposition~\ref{prop:free-com} \\
noncommutative associative with identity & $n!$ & $1$ & Proposition~\ref{prop:monoid} \\
$G=\{0,1\}$, $x*y$ is defined as
$1$, $\min\{x,y\}$, or $x+y \pmod 2$
& $1$ & $1$ & Example~\ref{ex:two-element-groupoids} \\

$G=\{0,1\}$, $x*y:=x$ & $n$ & $1$ & Example~\ref{ex:two-element-groupoids} \\
implication on $\{0,1\}$ & $n^{n-1}$ & $C_{n-1}$ & Proposition~\ref{prop:implication} \\
negated disjunction on $\{0,1\}$ & $D_{n-1}$ & $C_{n-1}$ & Proposition~\ref{prop:Sheffer} \\
$G=\{0,1\}$,
$x*y:=x+1\pmod2$
& 
\begingroup\renewcommand{\arraystretch}{1}
$\begin{array}[t]{ll}n & \text{for $n = 1, 2$} \\ 2n & \text{for $n \geq 3$} \end{array}$
\endgroup
& $2$ & Example~\ref{ex:two-element-groupoids} \\
arithmetic/geometric/harmonic mean & OEIS A007178 & $C_{n-1}$ & Proposition~\ref{prop:mean} \\
rock-paper-scissors operation & $D_{n-1}$ & $C_{n-1}$ & Proposition~\ref{prop:rock-paper-scissors} \\
commutative nonassociative with identity & $D_{n-1}$ & $C_{n-1}$ & Theorem~\ref{thm:com-neutral} \\
the cross product on $\mathbb{R}^3$ & $2 D_{n-1}$ & $C_{n-1}$ & Corollary~\ref{cor:cross} \\
Lie algebra with an $\mathfrak{sl}_2$-triple
over a field of characteristic not $2$
& $2 D_{n-1}$ & $C_{n-1}$ & Corollary~\ref{cor:Lie} \\
exponentiation on $\mathbb{R}_{\ge0}$ & $n^{n-1}$ & $C_{n-1}$ & Proposition~\ref{prop:exponentiation} \\
$a*b := a + e^{2 \pi i / k} b$ on $\CC$ & $k! S(n,k) + n \sum_{i=0}^{k-2} i! S(n-1,i)$ & $C_{k,n}$ & Theorem~\ref{thm:k-associative} \\
$a*b := -a-b$ on $\mathbb R$ & $(2^n - (-1)^n) / 3$ & $\left\lfloor 2^n/3 \right\rfloor$ & Theorem~\ref{thm:ominus} \\ 
\bottomrule
\end{tabular}
\end{center}

\bigskip
\caption{Summary of results}\label{tab:summary}
\end{table}


\section{Preliminaries}\label{sec:Prelim}

In this section we briefly recall some fundamental concepts concerning algebras, terms and identities, as well as trees, that are necessary for our work.
We will also introduce the \emph{\textup{(}fine\textup{)} associative\hyp{}commutative spectrum} of a groupoid.
This is a modification of the (fine) associative spectrum, introduced by Cs\'ak\'any and Waldhauser \cite{AssociativeSpectra1} and Liebscher and Waldhauser \cite{AssociativeSpectra2}.

Let $\IN_+$ denote the set of positive integers and let $\nset{n}$ denote the set $\{1, \dots, n\}$.

\subsection{Algebras, terms, and identities}
We recall some basic notions from universal algebra (see, e.g., Bergman \cite{Bergman}). 
An \emph{algebra}\footnote{This is not to be confused with an \emph{algebra over a field.} See Example~\ref{ex:algebras}\ref{ex:algebras:alg-over-field} and Section~\ref{sec:anticommutative}.} is a pair $\mathbf{A} = (A,F)$, where $A$ is a nonempty set (the \emph{universe} or the \emph{base set}) and $F = (f_i)_{i \in I}$ is a family of operations on $A$ (the \emph{fundamental operations}).
The mapping $\tau \colon I \to \IN$ that assigns to each $i \in I$ the arity of $f_i$ is called the \emph{algebraic similarity type} (or \emph{type}) of $\mathbf{A}$.
We may take as the index set a set $\mathcal{F}$ of \emph{operation symbols.}
Then an algebra of type $\tau \colon \mathcal{F} \to \IN$ is a pair $\mathbf{A} = (A, \mathcal{F}^\mathbf{A})$, where $\mathcal{F}^\mathbf{A} = (f^\mathbf{A})_{f \in \mathcal{F}}$ is a family of operations on $A$, each $f^\mathbf{A}$ having arity $\tau(f)$.
An algebra is \emph{trivial} if its universe has just one element.

Let us now fix an algebraic similarity type $\tau \colon \mathcal{F} \to \IN$.
Let $X$ be a set of \emph{variables} that is disjoint from $\mathcal{F}$.
We define \emph{terms} of type $\tau$ over $X$ by the following recursion:
every variable $x \in X$ is a term,
and if $f \in \mathcal{F}$ and $t_1, \dots, t_{\tau(f)}$ are terms, then $f(t_1, \dots, t_{\tau(f)})$ is a term.
Let $T_\tau(X)$ denote the set of all terms of type $\tau$ over $X$.
We usually take the set $X$ of variables to be one of the \emph{standard sets of variables:} for $n \in \IN_{+}$, $X_n := \{x_1, \dots, x_n\}$ and $X_\omega := \{x_1, x_2, \dots\}$.
Let $\var(t)$ denote the set of variables occurring in the term $t$, and let $\varcount{t}$ denote the total number of occurrences of variables in $t$.

We can define an algebra $\mathbf{T}_\tau(X)$ of type $\tau$ with universe $T_\tau(X)$ and, for each $f \in \mathcal{F}$, the $\tau(f)$\hyp{}ary fundamental operation $f^{\mathbf{T}_\tau(X)}$ defined as $f^{\mathbf{T}_\tau(X)}(t_1, \dots, t_{\tau(f)}) := f(t_1, \dots, t_{\tau(f)})$.
The algebra $\mathbf{T}_\tau(X)$ is called the \emph{term algebra} of type $\tau$ over $X$.
Note that  $\mathbf{T}_\tau(X)$ is \emph{absolutely free} over $X$, meaning that $\mathbf{T}_\tau(X)$ is generated by $X$ and has the \emph{universal mapping property,} i.e., for any algebra $\mathbf{A} = (A, \mathcal{F}^\mathbf{A})$ of type $\tau$ and any map $h \colon X \to A$, there is a unique homomorphism $h^\sharp \colon \mathbf{T}_\tau(X) \to \mathbf{A}$ that extends $h$.
In particular, the term algebra $\mathbf{T}_\tau(X_n)$ is called a \emph{free algebra of type $\tau$ on $n$ generators.}

Still in regard to the universal mapping property of $\mathbf{T}_\tau(X)$, let $\mathbf{A} = (A,\mathcal{F}^\mathbf{A})$ be an arbitrary algebra of type $\tau$, let $h \colon X \to A$ be an arbitrary map, and let $h^\sharp \colon \mathbf{T}_\tau(X) \to \mathbf{A}$ be the unique homomorphic extension of $h$.
We call the map $h$ an \emph{assignment} of values from $A$ for the variables in $X$.
For a term $t \in T_\tau(X)$, we call $h^\sharp(t)$ the \emph{value} of $t$ in $\mathbf{A}$ under $h$, and we say that $t$ \emph{evaluates} to $h^\sharp(t)$ in $\mathbf{A}$ under the assignment $h$.
For notational simplicity, we will write also $h(t)$ for $h^\sharp(t)$.
Of course, it is only the restriction of an assignment $h$ to $\var(t)$ that matters for the value $h^\sharp(t)$, so we may safely consider just partial assignments $\var(t) \to A$.

Let $\mathbf{A} = (A,\mathcal{F}^\mathbf{A})$ be an algebra of type $\tau$.
For each term $t \in T_\tau(X_n)$, we define an $n$\hyp{}ary operation $t^\mathbf{A}$ on $A$ by the following recursion.
If $t$ is a variable $x_i \in X_n$, then 
$t^\mathbf{A} = \pr_i^{(n)}$, where $\pr_i^{(n)}$ is the $i$-th $n$-ary \emph{projection}, defined by $(a_1, \dots, a_n) \mapsto a_i$.
If $t = f(t_1, \dots, t_{\tau(f)})$, where $f \in \mathcal{F}$ and $t_1, \dots, t_{\tau(f)} \in T_\tau(X_n)$, then $t^\mathbf{A}(a_1, \dots, a_n) := f^\mathbf{A}(t_1^\mathbf{A}(a_1, \dots, a_n), \dots, t_{\tau(f)}^\mathbf{A}(a_1, \dots, a_n))$.
The operation $t^\mathbf{A}$ is called the \emph{term operation} induced by $t$ on $\mathbf{A}$.
In other words, the term operation $t^\mathbf{A}$ provides the evaluations of $t$ in $\mathbf{A}$ under all possible assignments of values for variables:
if $h \colon X_n \to A$ is the assignment $h(x_i) = a_i$ for each $x_i \in X_n$, then $h(t) = t^\mathbf{A}(a_1, \dots, a_n)$.
For a set $T \subseteq T_\tau(X)$ of terms, let $T^\mathbf{A} := \{ t^\mathbf{A} \mid t \in T \}$.

An \emph{identity} of type $\tau$ is a pair $(s,t)$ of terms $s, t \in T_\tau(X_\omega)$, usually written as $s \approx t$.
An algebra $\mathbf{A} = (A,\mathcal{F}^\mathbf{A})$ \emph{satisfies} the identity $s \approx t$, denoted $\mathbf{A} \models s \approx t$, if $s^\mathbf{A} = t^\mathbf{A}$ (here we must assume that $s, t \in T_\tau(X_n)$ for some $n \in \IN$, but this is not a real restriction because it can always be done).
In other words, $\mathbf{A}$ satisfies $s \approx t$, if, by interpreting the terms $s$ and $t$ in the algebra $\mathbf{A}$, the two terms evaluate to the same value for all assignments of values from $A$ for the variables.
The set of all identities satisfied by an algebra $\mathbf{A}$ (by every member of a class $\mathcal{A}$ of algebras of type $\tau$, resp.)\ is called the \emph{equational theory} of $\mathbf{A}$ (of $\mathcal{A}$, resp.).
Conversely, the set of all algebras satisfying a set $\Sigma$ of identities is called the \emph{equational class} axiomatized by $\Sigma$.
It is well known that a class of algebras is an equational class if and only if it is a \emph{variety}, i.e., a class of algebras that is closed under homomorphic images, subalgebras, and direct products (Birkhoff~\cite{Birkhoff}).

In practice, we will often use the usual infix notation for binary operations, i.e., we write $a * b$ instead of $\mathord{*}(a,b)$.
For unary operations, it is also common to use special notations such as $a^{-1}$ instead of ${}^{-1}(a)$ or $\overline{a}$ instead of $\overline{\phantom{a}}(a)$.
We may also omit some parentheses from terms when there is no risk of misunderstanding.

\begin{example}\label{ex:algebras}
Examples of familiar algebras include the following.
\begin{enumerate}[label={\upshape(\roman*)}] 
\item
A \emph{groupoid} is an algebra $(G,{*})$ with a single binary operation $*$, i.e., an algebra of algebraic similarity type $\tau = (2)$.
A groupoid is \emph{commutative} if it satisfies the identity $x * y \approx y * x$ (\emph{commutative law}),
and it is \emph{associative} if it satisfies the identity $x * (y * z) \approx (x * y) * z$ (\emph{associative law}).
An associative groupoid is called a \emph{semigroup.}

\item
A \emph{monoid} is an algebra $(M,{*},e)$ of type $(2,0)$ that satisfies the associative law and the identities $x * e \approx e * x \approx x$. 
We call $e$ a \emph{neutral element} (or \emph{identity element}).

\item
A \emph{group} is an algebra $(G,{*},{}^{-1},e)$ of type $(2,1,0)$ such that $(G,{*},e)$ is a monoid and it satisfies the identities $x * x^{-1} \approx x^{-1} * x \approx e$.

\item
A \emph{ring} is an algebra $(R,{+}.{\cdot},{-},0)$ of type $(2,2,1,0)$ such that $(R,{+},{-},0)$ is a commutative group, $(R,{\cdot})$ is a semigroup, and it satisfies the identities $x \cdot (y + z) \approx (x \cdot y) + (x \cdot z)$ and $(y + z) \cdot x \approx (y \cdot x) + (z \cdot x)$ (\emph{distributive laws}).
A ring in which the multiplication is commutative is called a \emph{commutative ring.}

\item
A \emph{ring with identity} is an algebra $(R,{+},{\cdot},{-},0,1)$ of type $(2,2,1,0,0)$ such that $(R,{+},{\cdot},{-},0)$ is a ring and $(R,{\cdot},1)$ is a monoid.
A \emph{field} is a nontrivial commutative ring with identity in which every nonzero element has a multiplicative inverse (i.e., for all $x \in R \setminus \{0\}$, there is a $y \in R$ such that $x \cdot y = y \cdot x = 1$).

\item
Let $\mathbf{F} = (F,{+},{\cdot},{-},0,1)$ be a fixed field.
A \emph{vector space over $\mathbf{F}$} is an algebra
$(V, {+}, {-}, 0, (\lambda_a)_{a \in F})$
such that $(V, {+}, {-}, 0)$ is a commutative group and, for all $a, b \in F$, the following identities are satisfied:
\begin{align*}
\lambda_1(x) & \approx x, \\
\lambda_a(x + y) & \approx \lambda_a(x) + \lambda_a(y), \\
\lambda_{a+b}(x) & \approx \lambda_a(x) + \lambda_b(x), \\
\lambda_{a}(\lambda_b(x)) & \approx \lambda_{a \cdot b}(x).
\end{align*}
The unary operation $\lambda_a$ is called \emph{scalar multiplication by $a$.}
One often writes $a x$ for $\lambda_a(x)$.

\item\label{ex:algebras:alg-over-field}
Let $\mathbf{F}$ be a fixed field as above.
An \emph{algebra over the field $\mathbf{F}$} is an algebra $(A, {+}, {-}, 0, (\lambda_a)_{a \in F}, {*})$, where $(A, {+}, {-}, 0, (\lambda_a)_{a \in F})$ is a vector space over $\mathbf{F}$ and $*$ is a \emph{bilinear} binary operation, that is, for all $a, b \in F$, the following identities are satisfied:
\begin{align*}
(x + y) * z & \approx x * z + y * z, \\
z * (x + y) & \approx z * x + z * y, \\
\lambda_a(x) * \lambda_b(y) & \approx \lambda_{a \cdot b}(x * y).
\end{align*}
\end{enumerate}
\end{example}

\subsection{Associative-commutative spectrum}
In this paper, we discuss almost exclusively groupoids, i.e., algebras of type $\tau = (2)$.
Therefore, from now on, unless otherwise mentioned, we assume that the algebras under consideration are groupoids and the algebraic similarity type is not explicitly mentioned and is omitted from any notation we use.
When we speak, for example, of terms, we mean terms in the language of groupoids.
In fact, since we now have only one operation symbol, it does no harm to omit it from terms.

A term is called \emph{linear} if no variable occurs more than once therein.
We call a term $t \in T(X_n)$ \emph{full} over $X_n$ if $\var(t) = X_n$, i.e., every variable $x_i \in X_n$ occurs in $t$.
Thus, a full linear term over $X_n$ is a term that is obtained by inserting $n - 1$ pairs of parentheses in a valid way into the word $x_{\sigma(1)} x_{\sigma(2)} \dots x_{\sigma(n)}$ for some permutation $\sigma \in \SS_n$.
If, in the above, $\sigma$ is the identity permutation, then we get a \emph{bracketing} over $X_n$.
In other words, a full linear term $t \in T(X_n)$ is of the form $t = t'[x_{\sigma(1)}, \dots, x_{\sigma(n)}]$ for some bracketing $t'$ (the \emph{underlying bracketing} of $t$) and some permutation $\sigma \in \SS_n$.
(For a term $t \in T(X_n)$ and terms $u_1, \dots, u_n \in T(X_m)$, we let $t[u_1, \dots, u_n]$ denote the term in $T(X_m)$ obtained from $t$ by replacing each occurrence of $x_i$ by $u_i$, for all $i \in \nset{n}$.)
Let $F_n$ and $B_n$ denote the set of all full linear terms over $X_n$ and the set of all bracketings over $X_n$, respectively.
It is well known that the number of bracketings over $X_n$ equals the $(n-1)$\hyp{}st Catalan number $C_{n-1} = \frac{1}{n} \binom{2n-2}{n-1}$, i.e., $\card{B_n} = C_{n-1}$.
Consequently, $\card{F_n} = n! C_{n-1}$.
If $s, t \in B_n$ ($s, t \in F_n$, resp.)\ for some $n \in \IN$, then we call $s \approx t$ a \emph{bracketing identity} (a \emph{full linear identity,} resp.).

\begin{example}
There are twelve full linear terms over $X_3$, and they form the set
\[ 
F_3 = \{ ((x_{\sigma(1)} x_{\sigma(2)}) x_{\sigma(3)}) : \sigma \in \SS_3 \} \cup \{  (x_{\sigma(1)} (x_{\sigma(2)} x_{\sigma(3)})) : \sigma\in\SS_3 \}.
\]
There are two bracketings over $X_3$, and they form the set
\[
B_3 = \{ ((x_1 x_2) x_3), (x_1 (x_2 x_3)) \},
\]
which is a subset of $F_3$.
\end{example}

Let $\mathbf{G} = (G,*)$ be a groupoid.
The \emph{fine associative\hyp{}commutative spectrum} (in brief, \emph{fine ac\hyp{}spectrum}) of $\mathbf{G}$ is the sequence $(\finespac{\mathbf{G}})_{n \in \IN_{+}}$, where $\finespac{\mathbf{G}} := \{ (s,t) \in F_n \times F_n \mid \mathbf{G} \models s \approx t \}$.
In other words, $\finespac{\mathbf{G}}$ is the restriction of the equational theory of $\mathbf{G}$ to full linear identities over $X_n$.
It is clear that $\finespac{\mathbf{G}}$ is an equivalence relation on $F_n$.
The \emph{associative\hyp{}commutative spectrum} (in brief, \emph{ac\hyp{}spectrum}) of $\mathbf{G}$ is the sequence $(\spac{\mathbf{G}})_{n \in \IN}$, where $\spac{\mathbf{G}} := \card{F_n / \finespac{\mathbf{G}}}$, i.e., the number of equivalence classes of $\finespac{\mathbf{G}}$.
Equivalently, $\spac{\mathbf{G}}$ is the number of distinct term operations on $\mathbf{G}$ induced by the full linear terms over $X_n$, in symbols,
$\spac{\mathbf{G}} = \card{F_n^\mathbf{G}} = \card{\{ t^\mathbf{G} \mid t \in F_n \}}$.
The \emph{fine associative spectrum} $(\finespa{\mathbf{G}})_{n \in \IN_{+}}$ and the \emph{associative spectrum} $(\spa{\mathbf{G}})_{n \in \IN_{+}}$ of $\mathbf{G}$ were defined analogously by Liebscher and Waldhauser \cite{AssociativeSpectra2} by taking bracketings instead of full linear terms, i.e., by replacing $F_n$ with $B_n$ in the above definitions.
These numbers satisfy
$1 \leq \spa{\mathbf{G}} \leq \card{B_n} = C_{n-1}$ and $1 \leq \spac{\mathbf{G}} \leq \card{F_n} = n! C_{n-1}$.
We say $\mathbf{G}$ is \emph{totally nonassociative} if $\spa{\mathbf{G}} = C_{n-1}$ for all $n \geq 1$.
For notational simplicity, we may avoid giving a name to the groupoid and refer only to its fundamental operation, and hence we may sometimes replace $\mathbf{G}$ with $*$ in the notations introduced above.
For example, we may write $t^{*}$ for $t^\mathbf{G}$ or $\spac{*}$ for $\spac{\mathbf{G}}$.

The \emph{opposite groupoid} $(G,{*})^\mathrm{opp}$ of a groupoid $(G,{*})$ is the groupoid $(G,{\circ})$ with $a \circ b := b * a$ for all $a, b \in G$.
Groupoids $\mathbf{A}$ and $\mathbf{B}$ are said to be \emph{antiisomorphic} if $\mathbf{A}$ and $\mathbf{B}^\mathrm{opp}$ are isomorphic.
It is easy to see that isomorphic or antiisomorphic groupoids have the same associative spectrum and the same ac-spectrum.
The following facts follow immediately from the fact that equational classes of groupoids (classes axiomatized by identities) coincide with varieties of groupoids (classes closed under homomorphic images, subgroupoids, and direct products).
\begin{enumerate}[label={\upshape(\roman*)}]
\item If $\mathbf{A}$ is a homomorphic image of $\mathbf{B}$, then $\finespac{\mathbf{A}} \supseteq \finespac{\mathbf{B}}$ and $\spac{\mathbf{A}} \leq \spac{\mathbf{B}}$.
\item If $\mathbf{A}$ is a subgroupoid of $\mathbf{B}$, then $\finespac{\mathbf{A}} \supseteq \finespac{\mathbf{B}}$ and $\spac{\mathbf{A}} \leq \spac{\mathbf{B}}$.
\item If $\mathbf{C} = \mathbf{A} \times \mathbf{B}$, then $\finespac{\mathbf{C}} = \finespac{\mathbf{A}} \cap \finespac{\mathbf{B}}$ and $\spac{\mathbf{C}} \geq \max \{ \spac{\mathbf{A}}, \spac{\mathbf{B}} \}$.
\end{enumerate}

\subsection{Trees}\label{sec:trees}
Now recall that a (finite, undirected) graph is a pair $(V,E)$ where $V$ is a (finite) set whose elements are called \emph{vertices} and $E$ is a set of unordered pairs of vertices (i.e., sets of the form $\{a, b\}$, where $a, b \in V$) called \emph{edges}.
A graph is \emph{connected} if for any pair of vertices $u, v\in V$, there exists a sequence of vertices $(v_0, v_1, \ldots,v_\ell)$ such that $v_0=u$, $v_iv_{i+1}\in E$ for $i=0, 1,\ldots, \ell-1$, and $v_\ell=v$.
A \emph{cycle} is a sequence of distinct vertices $(v_0, v_1, \ldots,v_\ell)$ such that $v_iv_{i+1}\in E$ for all $i=0, 1, \ldots, \ell-1$ and $v_\ell v_0\in E$.
A \emph{tree} is a connected graph without any cycle.

Let $T$ be a \emph{rooted} tree, i.e., a tree with a distinguished vertex called the \emph{root} and with edges oriented away from the root.
We usually draw $T$ in such a way that its root appears at the very top and each vertex $v$ has a nonnegative number of children (i.e., out-neighbors) hanging below $v$.
A vertex in $T$ is a \emph{leaf} if it has no children, or an \emph{internal vertex} otherwise.

A rooted tree $T$ is \emph{ordered}%
\footnote{Ordered trees are often called \emph{plane trees} since a plane embedding of a tree induces a cyclic ordering of the neighbours of each vertex; moreover, if the root is drawn at the top -- following our drawing convention -- then the embedding specifies a linear ordering for the children of each internal vertex.}
if the children of each internal vertex are linearly ordered, or \emph{unordered} otherwise.
Given an ordered tree $T$, the unordered tree obtained from $T$ by simply ignoring the order of children of each internal vertex is called the \emph{underlying unordered tree} of $T$ and denoted by $T^\mathrm{u}$.
A rooted tree is \emph{labeled} if all of its vertices are labeled.
Given a vertex $v$ in a rooted tree $T$, the \emph{subtree of $T$ rooted at $v$} 
is the union of the paths from $v$ to the leaves that $v$ connects to, with $v$ itself as the root.

A rooted tree $T$ is a \emph{binary tree} if each internal vertex has exactly two children.
A binary tree is \emph{leaf-labeled} if its leaves are labeled.
The \emph{left subtree} $T_\mathrm{L}$ and \emph{right subtree} $T_\mathrm{R}$ of an ordered binary tree $T$ are the subtrees rooted at the left child and at the right child of the root of $T$, respectively.
If $S$ and $T$ are two ordered binary trees, then $S \wedge T$ is the ordered binary tree whose left and right subtrees are $S$ and $T$, respectively.
One can naturally extend these definitions to unordered binary trees by not distinguishing left and right.

Let $T$ be a binary tree with $n$ leaves labeled $1, \ldots, n$ in some order.
The \emph{left depth} $\delta_T(i)$, \emph{right depth} $\rho_T(i)$, and \emph{depth} $d_T(i)$ of a leaf $i$ in an ordered binary tree $T$ is the number of left, right, and all steps in the path from the root to the leaf labeled $i$.
This leads to the \emph{left depth sequence} $\delta_T:= (\delta_T(1), \dots, \delta_T(n))$, the \emph{right depth sequence} $\rho_T:= (\rho_T(1), \dots, \rho_T(n))$ and the \emph{depth sequence} $d_T := (d_T(1), \dots, d_T(n))$ of $T$.
If the leaves of $t$ are labeled $1, \dots, n$ from left to right, then each sequence above determines the ordered binary tree $T$ uniquely~\cite{AssociativeSpectra1, CatMod, DoubleMinus}.
The depth sequence can also be defined for unordered leaf-labeled binary trees.

There is a bijection between the set $T(X)$ of all terms over $X$ and the set of ordered binary trees with leaves labeled by the variables in $X$ (if $X=X_\omega$ then we may identify a label $x_i$ with $i$), defined recursively as follows:
each variable $x \in X$ corresponds to a tree with just one vertex labeled with $x$,
and
if the terms $t_1$ and $t_2$ correspond to trees $T_1$ and $T_2$, respectively, then the term $(t_1 t_2)$ corresponds to the tree $T_1 \wedge T_2$.
We write $T_t$ for the binary tree corresponding to the term $t$ via this bijection, and we write $t_T$ for the term corresponding to the binary tree $T$ via its inverse map.

\section{Free groupoids}\label{sec:Free}

In this section we show that the various upper bounds for the associative-commutative spectra mentioned in Section~\ref{sec:Intro} can be achieved by certain free groupoids with a \emph{small} number of generators.\footnote{We suspect that the free algebras (in a suitable variety) with countably infinitely many generators always have the largest spectrum possible (among the algebras in the variety).
}

\subsection{Free groupoid on one generator and free semigroup on two generators}
We first show that the upper bound $\spac{\mathbf{G}} \le n!C_{n-1}$ for the ac-spectrum $\spac{\mathbf{G}}$ of an arbitrary groupoid $\mathbf{G}$ becomes an equality when $\mathbf{G}$ is the free groupoid on one generator.

\begin{proposition}
\label{prop:free-groupoid}
For every $n \in \IN_{+}$, the $n$\hyp{}th term of the ac\hyp{}spectrum of the free groupoid $\mathbf{T}(X_1)$ with one generator is $\spac{\mathbf{T}(X_1)} = n!C_{n-1}$.
\end{proposition}

\begin{proof}
We need to show that distinct full linear terms over $X_n$ induce distinct term operations on $\mathbf{T}(X_1)$.
Let $t, u \in F_n$ with $t \neq u$.
If the underlying bracketings of $t$ and $u$ are distinct, say $t'$ and $u'$, then the assignment that maps every variable to $x_1$ yields terms $t'[x_1, \dots, x_1]$ and $u'[x_1, \dots, x_1]$, which are distinct terms in $T(X_1)$.
Assume now that the underlying bracketings of $t$ and $u$ are equal.
Then $t = b[x_{\sigma(1)}, \dots, x_{\sigma(n)}]$ and $u = b[x_{\tau(1)}, \dots, x_{\tau(n)}]$ for some bracketing $b \in B_n$ and permutations $\sigma$ and $\tau$ of $\nset{n}$, with $\sigma \neq \tau$.
There exists an $i \in \nset{n}$ such that $\sigma^{-1}(i) \neq \tau^{-1}(i)$.
Now the assignment $x_i \mapsto (x_1 x_1)$ and $x_j \mapsto x_1$ for all $j \neq i$ maps $t$ and $u$ to terms of the form $b[x_1, \dots, x_1, (x_1 x_1), x_1, \dots, x_1]$, where the single occurrence of $(x_1 x_1)$ appears at distinct positions (namely, $\sigma^{-1}(i)$ and $\tau^{-1}(i)$, respectively); these terms are clearly distinct.
This completes the proof.
\end{proof}

The set $\Sigma^\mathrm{a}_X$ of all identities over $X$ that are satisfied by all semigroups (the restriction of the equational theory of semigroups to identities over $X$) is a congruence of the term algebra $\mathbf{T}(X)$, and
the quotient $\mathbf{T}(X) / \Sigma^\mathrm{a}_X$ is referred to as a \emph{free semigroup} over $X$.
It is isomorphic to the semigroup $\mathbf{X}^{+} = (X^+, \cdot)$ of nonempty words over $X$ endowed with the operation $\cdot$ of concatenation.

If $\mathbf{G}$ is a semigroup then $\spac{\mathbf{G}} \leq n!$ for all $n \in \IN_+$, since all bracketings over $X_n$ induce the same term operation on $\mathbf{G}$ and it is only the order of variables in a full linear term that matters.
We show that the equality in this upper bound holds when $\mathbf{G}$ is the free semigroup with two generators.
\begin{proposition}
\label{prop:free-semigroup}
For every $n \in \IN_{+}$, the $n$\hyp{}th term of the ac\hyp{}spectrum of the free semigroup with two generators is  $\spac{\mathbf{T}(X_2) / \Sigma^\mathrm{a}_{X_2}} = n!$.
\end{proposition}

\begin{proof}
Let $t, u \in F_n$. Then $t = t'[x_{\sigma(1)}, \dots, x_{\sigma(n)}]$ and $u = u'[x_{\tau(1)}, \dots, x_{\tau(n)}]$ for some bracketings $t', u' \in B_n$ and permutations $\sigma, \tau \in \SS_n$.
We need to show that $t$ and $u$ induce distinct term operations on $\mathbf{X}_2^{+} = (X_2^+, \cdot)$ if and only if $\sigma \neq \tau$.
By associativity, the valuation of $t$ and $u$ under an assignment $h \colon X_n \to X_2^+$ results in the words $h(t) := h(x_{\sigma(1)}) \dots h(x_{\sigma(n)})$ and $h(u) := h(x_{\tau(1)}) \dots h(x_{\tau(n)})$, respectively.
Therefore, it is clear that the term operations coincide whenever $\sigma = \tau$.
If $\sigma \neq \tau$, then it is easy to find an assignment $h$ such that $h(t) \neq h(u)$ and hence the term operations are distinct.
For example, thinking of the two elements of $X_2$ as the binary digits $0$ and $1$, we define, for each $i \in \nset{n}$, $h(x_i)$ as the binary representation of the number $i$ with $N := \lceil \log_2 n \rceil$ bits.
Then $h(t)$ ($h(u)$, resp.)\ is a binary word of length $n N$, where the $i$\hyp{}th block of $N$ bits is the binary representation of $\sigma(i)$ ($\tau(i)$, resp.).
Since $\sigma \neq \tau$, there is an $i \in \nset{n}$ such that $\sigma(i) \neq \tau(i)$; therefore $h(t)$ and $h(u)$ differ in the $i$\hyp{}th block of $N$ bits.
\end{proof}

\begin{remark}
Note that we need (at least) two generators for the above proof to work.
In fact, the ac\hyp{}spectrum of the free semigroup with one generator is the constant $1$
sequence.
\end{remark}

\subsection{Free commutative groupoid with one generator}
The set $\Sigma^\mathrm{c}_X$ of all identities over $X$ that hold in all commutative groupoids (the restriction of the equational theory of commutative groupoids to identities over $X$) is a congruence of the term algebra $\mathbf{T}(X)$, and
the quotient $\mathbf{T}(X) / \Sigma^\mathrm{c}_X$ is referred to as a \emph{free commutative groupoid} over $X$.

It is clear that if $\mathbf{G}$ is a commutative groupoid and $s, t \in F_n$ are full linear terms such that $T_s$ and $T_t$ have the same underlying unordered binary tree, then $s^\mathbf{G} = t^\mathbf{G}$.
Consequently, $\spac{\mathbf{G}}$ is bounded above by the number $D_{n-1}$ of unordered binary trees with $n$ labeled leaves~\cite[A001147]{OEIS}, i.e., the solution to Schr\"oder's third problem; see, e.g., Stanley~\cite[p.~178]{EC2}.
To show that this upper bound can be achieved by taking $\mathbf{G}$ to be a free commutative groupoid, we need a slightly more complex construction (which we could in fact have used already for the free groupoids, but there a simpler method works).

\begin{definition}
For $n \in \IN_{+}$ and $i \in [n]$, let $\xi^{(n)}_i$ be the term
\[
((( \dots (((( \dots ( (x_1 x_1) x_1 ) \dots ) x_1 ) (x_1 x_1) ) x_1 ) \dots ) x_1 ) x_1),
\]
where the two subterms $(x_1 x_1)$ appear at depths $n$ and $i$; see Figure~\ref{fig:xi-i}.
\end{definition}

\begin{figure}
\tikzstyle{every node}=[inner sep=0.0625cm,circle, draw, fill=black]
\tikzstyle{nimi}=[draw=none,fill=none]
\begin{tikzpicture}
\node (0) at (0,0) {};
\node (0r) at ($(0) + (300:1)$) {};
\node (1) at ($(0) + (240:1)$) {};
\node (1r) at ($(1) + (300:1)$) {};
\node (i-2) at ($(1) + (240:1.5)$) {};
\node (i-2r) at ($(i-2) + (300:1)$) {};
\node (i-1) at ($(i-2) + (240:1)$) {};
\node (i-1r) at ($(i-1) + (300:1) + (0:1)$) {};
\node (i-1rl) at ($(i-1r) + (240:1)$) {};
\node (i-1rr) at ($(i-1r) + (300:1)$) {};
\node (i) at ($(i-1) + (240:1)$) {};
\node (ir) at ($(i) + (300:1)$) {};
\node (n-1) at ($(i) + (240:1.5)$) {};
\node (n-1r) at ($(n-1) + (300:1)$) {};
\node (n) at ($(n-1) + (240:1)$) {};
\node (nl) at ($(n) + (240:1)$) {};
\node (nr) at ($(n) + (300:1)$) {};
\draw ($(0) + (180:0.4)$) node[nimi]{$0$};
\draw ($(1) + (180:0.4)$) node[nimi]{$1$};
\draw ($(i-2) + (180:0.6)$) node[nimi]{$i-2$};
\draw ($(i-1) + (180:0.6)$) node[nimi]{$i-1$};
\draw ($(i) + (180:0.4)$) node[nimi]{$i$};
\draw ($(n-1) + (180:0.6)$) node[nimi]{$n-1$};
\draw ($(n) + (180:0.4)$) node[nimi]{$n$};
\foreach \u/\v in {0/1, 0/0r, 1/1r, i-2/i-2r, i-2/i-1, i-1/i-1r, i-1r/i-1rl, i-1r/i-1rr, i-1/i, i/ir, n-1/n-1r, n-1/n, n/nl, n/nr}
{
   \draw [thick] (\u) -- (\v);
}
\foreach \u/\v in {1/i-2, i/n-1}
{
   \draw [thick, dashed] (\u) -- (\v);
}
\foreach \u in {0r, 1r, i-2r, i-1rl, i-1rr, ir, n-1r, nl, nr}
{
   \draw ($(\u) + (270:0.3)$) node[nimi]{$x_1$};
}
\end{tikzpicture}

\caption{Binary tree corresponding to the term $\xi^{(n)}_i$.}
\label{fig:xi-i}
\end{figure}
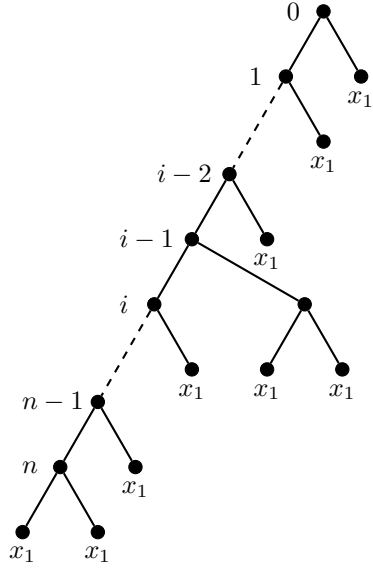

\begin{lemma}
The unordered binary trees corresponding to $\xi^{(n)}_i$ and $\xi^{(n)}_j$ are isomorphic if and only if $i = j$.
\end{lemma}

\begin{proof}
In $T_{\xi^{(n)}_\ell}$ with $1 \leq \ell \leq n-1$, the number of leaves at depth $n + 1$ is $2$, the number of leaves at depth $\ell + 1$ is $3$, and the number of leaves at depth $k$ with $k \in \nset{n} \setminus \{\ell\}$ is $1$.
In $T_{\xi^{(n)}_n}$, the number of leaves at depth $n + 1$ is $4$ and the number of leaves at depth $k$ with $k \in \nset{n}$ is $1$.
Consequently, $T_{\xi^{(n)}_i}$ and $T_{\xi^{(n)}_j}$ are not isomorphic if $i \neq j$.
If $i = j$, then $T_{\xi^{(n)}_i}$ and $T_{\xi^{(n)}_j}$ are equal and hence obviously isomorphic.
\end{proof}

\begin{proposition}\label{prop:free-com}
For every $n \in \IN_{+}$, the $n$\hyp{}th term of the ac\hyp{}spectrum of the free commutative groupoid $\mathbf{T}(X_1) / \Sigma^\mathrm{c}_{X_1}$ with one generator is $D_{n-1}$.
\end{proposition}

\begin{proof}
We need to show that for all $t, u \in F_n$ such that $t \approx u \notin \Sigma^\mathrm{c}_{X_n}$, the term operations on $\mathbf{T}(X_1) / \Sigma^\mathrm{c}_{X_1}$ induced by $t$ and $u$ are distinct.
Let $t, u \in F_n$ with $t \approx u \notin \Sigma^\mathrm{c}_{X_n}$.
Consider the assignment $x_i \mapsto \xi^{(n)}_n$.
The resulting terms $t' := t[\xi^{(n)}_1, \dots, \xi^{(n)}_n]$ and $u' := u[\xi^{(n)}_1, \dots, \xi^{(n)}_n]$ belong to $T(X_1)$.
It is easy to verify that the labeled unordered binary trees corresponding to $t'$ and $u'$ are isomorphic if and only if those of $t$ and $u$ are.
Since $t \approx u \notin \Sigma^\mathrm{c}_{X_n}$, the binary trees corresponding to $t$ and $u$ are nonisomorphic as labeled unordered binary trees; hence so are the trees corresponding to $t'$ and $u'$, and we have $t' \approx u' \notin \Sigma^\mathrm{c}_{X_1}$, so ${t'}^{\mathbf{T}(X_1) / \Sigma^\mathrm{c}_{X_1}} \neq {u'}^{\mathbf{T}(X_1) / \Sigma^\mathrm{c}_{X_1}}$.
\end{proof}

\section{Associative or commutative groupoids}\label{sec:associative-commutative}

In this section we study the ac\hyp{}spectra of some groupoids that are either associative or commutative.

\subsection{Associative groupoids (semigroups)}\label{sec:associative}
Assume first that $\mathbf{G}$ is an associative groupoid (a semigroup).
We know that the upper bound $\spac{\mathbf{G}} \leq n!$ for semigroups is reached by the free semigroup with two generators (Proposition \ref{prop:free-semigroup}).
Now we provide another example of a family of groupoids for which this upper bound is achieved.


\begin{proposition}\label{prop:monoid}
If $\mathbf{G} = (G,{*})$ is a noncommutative monoid, then we have $\spac{\mathbf{G}} = n!$ for all $n \geq 1$.
\end{proposition}

\begin{proof}
Let $s, t \in F_n$.
Then $s = s'[x_{\pi(1)}, \dots, x_{\pi(n)}]$ and $t = t'[x_{\tau(1)}, \dots, x_{\tau(n)}]$ for some bracketings $s', t' \in B_n$ and permutations $\pi, \tau \in \SS_n$.
We show that $s^\mathbf{G} = t^\mathbf{G}$ if and only if $\pi = \tau$, from which the proposition follows.

If $\pi = \tau$, then clearly $s^\mathbf{G} = t^\mathbf{G}$ because the bracketings $s'$ and $t'$ are irrelevant by the associativity of $*$.
Assume now that $\pi \neq \tau$.
Then $\tau^{-1} \pi$ is a nonidentity permutation, and therefore it must have an inversion pair, say $(i,j)$ with $i < j$ and $\tau^{-1} \pi(i) > \tau^{-1} \pi(j)$.
Since $*$ is not commutative, there exist elements $a, b \in G$ such that $a * b \neq b * a$.
Let $e$ be the neutral element of $*$.
Assign the value $a$ to $x_{\pi(i)} = x_{\tau(\tau^{-1} \pi(i))}$, the value $b$ to $x_{\pi(j)} = x_{\tau(\tau^{-1} \pi(j))}$, and the value $e$ to all remaining variables.
Under this assignment, the term $s$ evaluates to $e \dots e a e \dots e b e \dots e = ab$, where on the left side, $a$ and $b$ occur as the $i$\hyp{}th and the $j$\hyp{}th factors, respectively.
On the other hand, the term $t$ evaluates to $e \dots e b e \dots e a e \dots e = ba$, where on the left side, $a$ and $b$ occur as the $\tau^{-1} \pi(i)$\hyp{}th and the $\tau^{-1} \pi(j)$\hyp{}th factors, respectively.
This shows that $s^\mathbf{G} \neq t^\mathbf{G}$.
\end{proof}

The above proposition would no longer be true if we omitted the assumption that the groupoid has a neutral element, as shown by the following example.

\begin{example}\label{ex:two-element-groupoids}
Cs\'{a}k\'{a}ny and Waldhauser~\cite[Section~4]{AssociativeSpectra1} determined the associative spectrum of every two\hyp{}element groupoid.
Such a groupoid is isomorphic or antiisomorphic to $\mathbf{G} = (\{0,1\},*)$, where $x*y$ is defined as one of the following: (1) $1$, (2) $x$, (3) $\min\{x,y\}$, (4) $x+y \pmod 2$, (5) $x+1 \pmod 2$, (6) $x \NOR y$
(negated disjunction, NOR) or (7) $x \rightarrow y$ (implication).
We now set out to determine their ac\hyp{}spectra.

We have $\spac{\mathbf{G}} = 1$ for all $n \in\IN_+$ if $*$ defined by (1), (3), or (4), since $*$ is both associative and commutative in these three cases.
The operation $*$ defined by (2) is associative but not commutative, and we have $\spac{\mathbf{G}} = n$ for all $n \in\IN_+$, since for any $t \in F_n$ with $t = t'[x_{\tau(1)}, \dots, x_{\tau(n)}]$, where $t' \in B_n$ and $\tau \in \SS_n$, we have $t^\mathbf{G} = \pr^{(n)}_{\tau(1)}$.
For the operation $*$ defined by (5), we have $\spac[1]{\mathbf{G}} = 1$, $\spac[2]{\mathbf{G}} = 2$, and $\spac{\mathbf{G}} = 2n$ for all $n \geq 3$, since for any $t \in F_n$ with $t = t'[x_{\tau(1)}, \dots, x_{\tau(n)}]$, where $t' \in B_n$ and $\tau \in \SS_n$, we have $t^\mathbf{G}(a_1, \dots, a_n) = a_{\tau(1)} + d \pmod 2$, where $d$ is the left depth of the leftmost leaf in the binary tree corresponding to $t$.
The groupoids given by (6) and (7) are totally nonassociative and their ac-spectra will be determined in Section~\ref{sec:TN}.
\end{example}

\subsection{Commutative groupoids}\label{sec:commutative}
Now we assume that $\mathbf{G} = (G,{*})$ is a commutative groupoid.
Recall that we have an upper bound $\spac{\mathbf{G}} \le D_{n-1}$ which is attained by the free commutative groupoid with one generator (Proposition \ref{prop:free-com}).
As shown by the following lemma, any commutative groupoid $\mathbf{G}$ reaching this upper bound must be totally nonassociative.

\begin{lemma}
\label{lem:com:large-implies-na}
Let $\mathbf{G} = (G,{*})$ be a commutative groupoid.
If $\spac{\mathbf{G}} = D_{n-1}$ for $n \in \IN_{+}$, then $\mathbf{G}$ is totally nonassociative, i.e., $\spa{\mathbf{G}} = C_{n-1}$ for all $n \in \IN_+$.
\end{lemma}

\begin{proof}
The commutativity of $*$ and the hypothesis $\spac{\mathbf{G}} = D_{n-1}$ imply that terms $s, t \in F_n$ induce the same operation on $\mathbf{G}$ if and only if $s$ and $t$ are are congruent modulo the equational theory $\Sigma^\mathrm{c}_{X_n}$ of commutative semigroups, i.e., the unordered binary trees $T_s^\mathrm{u}$ and $T_t^\mathrm{u}$ with labeled leaves corresponding to the terms $s$ and $t$ are isomorphic.
Binary trees corresponding to bracketings of $n$ variables are isomorphic as unordered leaf-labeled binary trees if and only if the bracketings are equal.
Consequently, $\spa{\mathbf{G}} = C_{n-1}$.
\end{proof}

The converse of Lemma~\ref{lem:com:large-implies-na} does not hold.
If $*$ is the arithmetic, geometric, or harmonic mean, then $\spa{*} = C_{n-1}$ for all $n \geq 1$ (see Cs\'ak\'any, Waldhauser~\cite{AssociativeSpectra1}).
However, as we are going to show next, its
ac\hyp{}spectrum agrees with an interesting sequence in OEIS~\cite[A007178]{OEIS}, which enumerates different ways to write $1$ as an ordered sum of $n$ powers of $2$ (i.e., \emph{compositions of $1$ into powers of $2$}) and is also related to the so\hyp{}called \emph{prefix codes} or \emph{Huffman codes} (see, e.g., Even and Lempel~\cite{EvenLempel}, Giorgilli and Molteni~\cite{GM}, Knuth~\cite{Knuth}, Krenn and Wagner~\cite{KrennWagner} and Lehr, Shallit and Tromp~\cite{LST}).

\begin{proposition}\label{prop:mean}
If $*$ is the arithmetic mean on $\RR$, the geometric mean on $\RR_{+}$, or the harmonic mean on $\RR_{+}$, then $\spac{*}$ equals the number of ways to write $1$ as an ordered sum of $n$ powers of $2$ for all $n \geq 1$.
\end{proposition}

\begin{proof}
As observed by Cs\'{a}k\'{a}ny and Waldhauser~\cite{AssociativeSpectra1}, the groupoid $\RR$ with the arithmetic mean $(x+y)/2$ is isomorphic to the groupoid $\RR_{+}$ with the geometric mean $\sqrt{xy}$ via $x \mapsto e^x$, and the groupoid $\RR_{+}$ with arithmetic mean $(x+y)/2$ is isomorphic to the groupoid $\RR_{+}$ with the harmonic mean $( ( x^{-1}+y^{-1} ) /2 )^{-1}$ via $x \mapsto x^{-1}$.
Therefore we may assume $\mathbf{G} = (\RR,{*})$ with $x*y := (x+y)/2$.

A bracketing $t \in B_n$ induces the operation $t^\mathbf{G}(a_1, \dots, a_n) = \sum_{i=1}^n 2^{-d_i} a_i$ over the arithmetic mean, where $d_i$ is the depth of the $i$\hyp{}th leaf in $T_t$.
Thus two bracketings over $X_n$ induce the same operation on $\mathbf{G}$ if and only if their corresponding binary trees have the same depth sequence (which means that the trees coincide).
It follows that the operations on $\mathbf{G}$ induced by the full linear terms over $X_n$ are in a one\hyp{}to\hyp{}one correspondence with
the sequences belonging to the union of the $\SS_n$\hyp{}orbits of all depth sequences of binary trees with $n$ leaves.
It is known and can be shown by induction that any sequence $(d_1, \dots, d_n)$ in these orbits must satisfy the condition $1 = 2^{-d_1} + \dots + 2^{-d_n}$, and any sequence of positive integers satisfying this condition must be a permutation of the depth sequence of a binary tree.
The result follows.
\end{proof}

Next, we consider the \emph{rock\hyp{}paper\hyp{}scissors operation} $*$ defined on the set $\{\text{rock}, \text{paper}, \text{scissors}\}$ by $x * y = y * x := x$ if $x$ beats $y$ or $x=y$.
Rock beats scissors, scissors beat paper, and paper beats rock.

\begin{proposition}\label{prop:rock-paper-scissors}
For the rock\hyp{}paper\hyp{}scissors operation $*$, we have $\spac{*} = D_{n-1}$ and $\spa{*} = C_{n-1}$ for all $n \geq 1$.
\end{proposition}

\begin{proof}
By Lemma~\ref{lem:com:large-implies-na}, it suffices to prove $\spac{*} = D_{n-1}$ for all $n \geq 1$.
We proceed by induction on $n$.
Let $s, t \in F_n$.
Since $*$ is commutative, it suffices to show that if $T_s^\mathrm{u} \neq T_t^\mathrm{u}$ are distinct then $s^{*} \neq t^{*}$.
The claim clearly holds for $1 \leq n \leq 2$.
Let now $n \geq 3$, and assume that the claim holds for terms with fewer than $n$ variables.
Let $x_j$ and $x_k$ be two leaves with a common parent in $T_s$.

First, suppose that $x_j$ and $x_k$ also share a parent in $T_t$.
Then deleting $x_j$ and $x_k$ from both $T_s$ and $T_t$ and labeling their parent with a new variable will yield binary trees $U_s$ and $U_t$ with $n-1$ labeled leaves.
Since $T_s^\mathrm{u} \neq T_t^\mathrm{u}$ implies $U_s^\mathrm{u} \neq U_t^\mathrm{u}$, we have $(t_{U_s})^{*} \neq (t_{U_t})^{*}$ by the induction hypothesis.
It follows that $s^{*} \neq t^{*}$.

Next, suppose that $x_j$ and $x_k$ do not share a parent in $T_t$.
Let $a$ be the first common ancestor of $x_j$ and $x_k$ in $T_t$, let $S$ be the subtree of $T_t$ rooted at $a$, and let $R$ and $R'$ be the subtrees of $T_t$ rooted at the two children of $a$.
We may assume, without loss of generality, that $R$ contains $x_j$ and at least one other leaf, and $R'$ contains $x_k$.
Then using an assignment $h$ with $h(x_j) = \text{rock}$, $h(x_k) = \text{scissors}$, and $h(x_i) = \text{paper}$ for all $i \in \nset{n} \setminus \{j,k\}$, we have $h(t_R) = \text{paper}$, $h(t_{R'}) = \text{scissors}$, and thus $h(t) = \text{scissors}$.
On the other hand, we have $h(s) = \text{paper}$.
Therefore $s^{*} \neq t^{*}$.
\end{proof}

\subsection{Commutative groupoids with a neutral element}\label{sec:commutative-neutral}
For a commutative groupoid $\mathbf{G}$ with a neutral element, we show that there are only two distinct possibilities: $\mathbf{G}$ is either associative 
or totally nonassociative (see Theorem~\ref{thm:com-neutral} below). 
Examples of nonassociative commutative groupoids with a neutral element include the \emph{Jordan algebras} of $n \times n$ self\hyp{}adjoint matrices over $\RR, \CC$, or $\mathbb{H}$ (the algebra of quaternions) with a product defined by $x\circ y := (xy+yx)/2$. 
The identity matrix $I_n$ is the neutral element for this commutative algebra.

\begin{theorem}
\label{thm:com-neutral}
Let $\mathbf{G} = (G,{*})$ be a commutative groupoid with neutral element $e$.
Then either
\begin{enumerate}[label={\upshape(\roman*)}]
\item $\mathbf{G}$ is associative, in which case $\spa{\mathbf{G}} = \spac{\mathbf{G}} = 1$ for all $n \in \IN_{+}$, or
\item $\spa{\mathbf{G}} = C_{n-1}$ and $\spac{\mathbf{G}} = D_{n-1}$ for all $n \geq 1$.
\end{enumerate}
\end{theorem}

\begin{proof}
If $\mathbf{G}$ is associative and commutative, then all full terms in $F_n$ induce the same term operation on $\mathbf{G}$ and so $\spa{\mathbf{G}} = \spac{\mathbf{G}} = 1$ for all $n \in \IN_{+}$.

Assume now that $\mathbf{G}$ is not associative.
Then there exist elements $a, b, c \in G$ such that $a(bc) \neq (ab)c$.
By Lemma~\ref{lem:com:large-implies-na}, it suffices to prove $\spac{\mathbf{G}} = D_{n-1}$ for all $n \geq 1$.
We are going to show that $s^\mathbf{G} \neq t^\mathbf{G}$ whenever $s, t \in F_n$ are noncongruent terms modulo the equational theory $\Sigma^\mathrm{c}_{X_n}$ of commutative groupoids, i.e., $T_s^\mathrm{u} \neq T_t^\mathrm{u}$.
The claim holds trivially for $n=1$ and $n=2$, so we may assume $n \geq 3$.

Consider now the case $n = 3$.
Modulo $\Sigma^\mathrm{c}_{X_3}$, there are three distinct full linear terms over $X_3$, namely
$t_1 := x_1 (x_2 x_3)$, $t_2 := x_2 (x_1 x_3)$, $t_3 := x_3 (x_1 x_2)$,
and the term functions they induce on $\mathbf{G}$ are distinct because
\begin{align*}
& t_1^\mathbf{G}(a,c,b) = a(cb) = a(bc) \neq (ab)c = c(ab) = t_2^\mathbf{G}(a,c,b), \\
& t_1^\mathbf{G}(a,b,c) = a(bc) \neq (ab)c = c(ab) = t_3^\mathbf{G}(a,b,c), \\
& t_2^\mathbf{G}(b,a,c) = a(bc) \neq (ab)c = c(ba) = t_3^\mathbf{G}(b,a,c).
\end{align*}

Assume now that $n \geq 4$ and that the claim holds for terms with fewer than $n$ variables.
We have $s = (s_L s_R)$ and $t = (t_L t_R)$ for subterms $s_L$, $s_R$, $t_L$, $t_R$.
Note that the term $t' := (t_R t_L)$ is congruent to $t$ modulo $\Sigma^\mathrm{c}_{X_n}$, so we may consider $t'$ in place of $t$ if necessary.
Observe that $\{\var(s_L), \var(s_R)\}$ and $\{\var(t_L), \var(t_R)\}$ are partitions of $\nset{n}$.

Consider first the case when $\{\var(s_L), \var(s_R)\} \neq \{\var(t_L), \var(t_R)\}$.
By taking, if necessary, the term $t'$ in place of $t$, and by changing the roles of $s$ and $t$, if necessary, it follows that there exist $j, k, \ell \in \nset{n}$ such that $x_j \in \var(s_L) \cap \var(t_L)$, $x_k \in \var(s_R) \cap \var(t_L)$, and $x_\ell \in \var(s_R) \cap \var(t_R)$.
Let $h$ be the assignment $x_j \mapsto a$, $x_k \mapsto b$, $x_\ell \mapsto c$, and $x_i \mapsto e$ for all $i \in \nset{n} \setminus \{i, j, k\}$.
Then $h(s) = h(s_L) * h(s_R) = a(bc)$ and $h(t) = h(t_L) * h(t_R) = (ab)c$, which shows that $s^\mathbf{G} \neq t^\mathbf{G}$.

Finally, consider the case when $\{\var(s_L), \var(s_R)\} = \{\var(t_L), \var(t_R)\}$.
By taking, if necessary, the term $t'$ in place of $t$, we may assume that $\var(s_L) = \var(t_L)$ and $\var(s_R) = \var(t_R)$.
We must have $s_L \neq t_L$ or $s_R \neq t_R$, say the former.
Then $s_L^\mathbf{G} \neq t_L^\mathbf{G}$ by the induction hypothesis.
Hence there is an assignment $h$ for the variables in $\var(s_L) = \var(t_L)$ such that $h(s_L) \neq h(t_L)$.
Extend $h$ into an assignment $h'$ on $X_n$ by defining $x_i \mapsto e$ for all $x_i \in \var(s_R) = \var(t_R)$.
Then
\[ h'(s) = h'(s_L) * h'(s_R) = h(s_L) * e = h(s_L) \neq h(t_L) = h(t_L) * e = h'(t_L) * h'(t_R) = h'(t) \]
which shows that $s^\mathbf{G} \neq t^\mathbf{G}$.
This concludes the proof.
\end{proof}

\section{Anticommutative algebras}\label{sec:anticommutative}

We now turn our attention to ac\hyp{}spectra of bilinear products in algebras over a field.
An algebra over a field $\FF$ of characteristic not $2$ is said to be \emph{anticommutative} if it satisfies the identity $xy \approx -yx$, which implies the identity $xx \approx 0$ since $xx \approx -xx$.

\subsection{Commutative version of a bilinear product}
Given an anticommutative algebra over a field, we can turn the product $*$ into a commutative bilinear product $\circledast$ as follows.

\begin{definition}
Let $\mathbf{A} = (A, {+}, {-}, 0, (\lambda_a)_{a \in F}, {*})$ be an algebra over a field $\FF$.
Let $g$ be any choice function on the collection $C := \{ \{ a, -a \} \mid a \in A \}$ and let $f \colon A \to C$, $f(a) := \{a, -a\}$.
(Recall that a \emph{choice function} on a collection $C$ of subsets of some base set $X$ is a mapping $g \colon C \to X$ such that $g(S) \in S$ for every $S \in C$.)
Note that any map $f$ arising in this way is \emph{even,} i.e., it satisfies $f(a) = f(-a)$ for all $a \in A$.
Now we can fix a basis $B$ of the vector space $A$, and for all basis vectors $a, b \in B$, we define $a \circledast b := g(f(a * b))$.
This partial operation extends to a commutative bilinear product on $A$.
(It is well known that any partial operation on $A$ with domain $B$ extends in a unique way to a bilinear product on $A$, and a bilinear product is commutative if and only if its restriction to the basis is commutative.)
Such a product $\circledast$ will be referred to as a \emph{commutative version} of $*$.
\end{definition}

\begin{theorem}\label{thm:anticommutative}
Let $\mathbf{A} = (A, {+}, {-}, 0, (\lambda_a)_{a \in F}, {*})$ be an anticommutative algebra over a field $\FF$, and assume that the commutative version $\circledast$ of $*$ satisfies $s^{\circledast} \ne \pm t^\circledast$ for any terms $s,t\in F_n$ with $T^\mathrm{u}_s \ne T^\mathrm{u}_t$.
\begin{enumerate}[label={\upshape(\roman*)}]
\item
We have $s^{\mathrm{ac}}_n(\circledast) = D_{n-1}$ and $\spa{\circledast} = C_{n-1}$ for all $n\ge1$.
\item
We have $s^{\mathrm{ac}}_n(*) = 2s^{\mathrm{ac}}_n(\circledast) = 2D_{n-1}$ for all $n\ge2$ and $\spa{*} = C_{n-1}$ for all $n\ge1$.
\end{enumerate}
\end{theorem}

\begin{proof}
Let $A$ be the universe of the algebra.
Since $s^{\circledast} \ne \pm t^\circledast$ for any terms $s,t\in F_n$ with $T^\mathrm{u}_s \ne T^\mathrm{u}_t$, the ac-spectrum $\spac{\circledast}$ of the commutative operation $\circledast$ must attain the upper bound $D_{n-1}$ for commutative operations.
This implies $\spa{\circledast} = C_{n-1}$ for all $n\ge1$ by Lemma~\ref{lem:com:large-implies-na}.

Assume $n\ge2$ below.
For any $t \in F_n$, if $(t_1 t_2)$ is a subterm of $t$ and $\widetilde{t}$ is the term obtained from $t$ by replacing $(t_1 t_2)$ by $(t_2 t_1)$, then $\widetilde{t}^{\circledast} = t^{\circledast}$ and $\widetilde{t}^* = -t^*$.
Any terms $s, t \in F_n$ such that $T_s^\mathrm{u} = T_t^\mathrm{u}$ can be obtained from each other by a sequence of such swaps of subterms.
On the other hand, for any terms $s, t \in F_n$ such that $T^\mathrm{u}_s \ne T^\mathrm{u}_t$, we have $s^{\circledast} \ne \pm t^\circledast$ by the hypothesis.

It follows that for each $f \in F_n^{\circledast} := \{t^{\circledast} \mid t \in F_n\}$, the terms in $F_n$ that give rise to $f$ induce two distinct functions on $(A,*)$, namely, $f$ and $-f$, and for distinct $f, g \in F_n^{\circledast}$, any terms in $F_n$ giving rise to $f$ and $g$ induce distinct functions on $(A,*)$.
Therefore the cardinality of $F_n^* := \{t^* \mid t \in F_n\}$ is exactly two times that of $F_n^{\circledast}$, i.e., $\spac{*} = 2 \spac{\circledast}$ for all $n\ge2$.

Finally, let $s$ and $t$ be distinct bracketings. 
Let $T_s$ and $T_t$ be the corresponding ordered binary trees with leaves labeled by $x_1,\ldots,x_n$ from left to right.
Then the underlying unordered leaf-labeled
trees $T_s^{\mathrm{u}}$ and $T_t^{\mathrm{u}}$ must be distinct. 
Therefore $s^\circledast \neq t^\circledast$ by our assumption, and hence $s^* \neq t^*$ by what we have shown above.
This shows that $\spa{*} = C_{n-1}$ for all $n\ge1$. 
\end{proof}

To apply Theorem~\ref{thm:anticommutative} to some familiar anticommutative algebras, we need a lemma.

\begin{lemma}\label{lem:anticommutative}
Let $\mathbf{A} = (A, {+}, {-}, 0, (\lambda_a)_{a \in F}, {*})$ be an anticommutative algebra over a field $\mathbb{F}$, and let $\circledast$ be the commutative version of $*$.
Suppose the following conditions hold for some $U \subseteq A \setminus \{0\}$ and $P \subseteq U \times U$.
\begin{enumerate}[label={\upshape(\roman*)}]
\item\label{lem:anticommutative:i}
For any $w\in U$, there exists $u\in U$ such that $(u,w)\in P$,
\item\label{lem:anticommutative:ii}
There exists $(u,v)\in P$ such that $(c(u\circledast v), w)\in P$ for some scalar $c\ne0$ and some $w\in U$.
\item\label{lem:anticommutative:iii}
There exists $(u,v)$ and $(u,w)$ in $P$ such that $(c(v\circledast w),z)\in P$ for some scalar $c\ne0$ and some $z\in U$.
\item\label{lem:anticommutative:iv}
For any $n \geq 2$, $t \in F_n$, $j \in \nset{n}$, and $(u,w)\in P$, there exist $u_1,\ldots,u_n\in U$ such that $u_j = u$ and $t^{\circledast}(u_1, \ldots, u_n) = cw$ for some scalar $c\ne0$. 
\end{enumerate}
Then $s^{\circledast} \ne \pm t^\circledast$ for any terms $s, t \in F_n$ with $T^\mathrm{u}_s \neq T^\mathrm{u}_t$.
\end{lemma}

\begin{proof}
We will prove the slightly stronger statement that for any linear terms $s, t \in T(X_\omega)$ with $\var(s) = \var(t)$ and $T_s^\mathrm{u} \ne T_t^\mathrm{u}$, there is an assignment $h' \colon \var(s)=\var(t) \to U$ such that one of $h'(s)$ and $h'(t)$ is $0$ and the other is nonzero.
Consequently, $s^{\circledast} \ne \pm t^\circledast$ for any terms $s, t \in F_n$ with $T^\mathrm{u}_s \neq T^\mathrm{u}_t$.

We proceed by induction on $n$.
We must have $n \geq 3$ since $T_s^\mathrm{u} \ne T_t^\mathrm{u}$.
For $n = 3$, we have, without loss of generality, $s = (x_1 x_2) x_3$ and $t = x_1 (x_2 x_3)$.
Let $u$, $v$, $w$, and $c$ be as in \ref{lem:anticommutative:ii}, i.e., $(u,v) \in P$ and $(c (u \circledast v), w) \in P$; note that the latter implies that $u \circledast v \in U$.
By \ref{lem:anticommutative:iv}, there exists $z \in U$ such that $u \circledast z = v$.
Then $s^\circledast(u,u,z) = (u \circledast u) \circledast z = 0 \circledast z = 0$ and
$t^\circledast(u,u,z) = u \circledast (u \circledast z) = u \circledast v \in U$.

Assume now that $n > 3$ and that the lemma holds for linear terms with fewer than $n$ variables.
There exist two leaves labeled by $x_j$ and $x_k$ with a common parent in $T_s$.
We distinguish cases according to whether the leaves with these labels have a common parent also in $T_t$.

\vskip3pt\noindent
\textsf{Case 1}: The leaves labeled by $x_j$ and $x_k$ also share a parent in $T_t$.
Then deleting $x_j$ and $x_k$ from both $T_s$ and $T_t$ and labeling their parent with a new variable $y$ will result in two binary trees $R_s$ and $R_t$ with $n-1$ labeled leaves.
Since $T_s^\mathrm{u} \neq T_t^\mathrm{u}$ implies $R_s^\mathrm{u} \neq R_t^\mathrm{u}$, the inductive hypothesis provides an assignment $h \colon \var(t_{R_t}) \to U$ such that one of $h(t_{R_s})$ and $h(t_{R_t})$ is $0$ and the other is nonzero.
Now let $h' \colon \var(t) \to U$ be an assignment that coincides with $h$ on $\var(t_{R_t}) \setminus \{y\}$ and assign values to $x_j$ and $x_k$ such that $h'(x_j) \circledast h'(x_k) = c h(y)$ for some $c\ne0$; such values exist by \ref{lem:anticommutative:i} and \ref{lem:anticommutative:iv} with $n=2$ and $w=h(y)$.
It follows that $h'(s) = ch(t_{R_s})$ and $h'(t) = ch(t_{R_t})$, so one of $h'(s)$ and $h'(t)$ is zero and the other is nonzero.

\vskip3pt\noindent
\textsf{Case 2}: The leaves labeled by $x_j$ and $x_k$ do not share a parent in $T_t$.
Let $S$ be the subtree of $T_t$ rooted at the first common ancestor $a$ of $x_j$ and $x_k$, and let $R$ and $R'$ be the subtrees of $T_t$ rooted at the two children of $a$.
We may assume, without loss of generality, that $R$ contains $x_j$ and $R'$ contains $x_k$.

\vskip3pt\noindent
\textsf{Case 2.1}: Either $R$ or $R'$, say the former, contains only one leaf.
Then the latter must have at least two leaves.
Let $u$, $v$, and $w$ be as in \ref{lem:anticommutative:ii}.
By \ref{lem:anticommutative:iv}, there exists an assignment $h \colon \var(t_S) \to U$ such that $h(x_j) = h(x_k) = u$, $h(t_R) = u$, $h(t_{R'}) = cv$, and $h(t_S) =c (u \circledast v) \in U$ for some nonzero scalar $c$.
If $S=T_t$ then let $h' := h$.
If $S\ne T_t$ then we can use \ref{lem:anticommutative:iv} again to obtain an assignment $h' \colon \var(t) \to U$ extending $h$ such that $h'(t) = c c' w$ for some nonzero scalar $c'$.
On the other hand, we have $h'(s) = 0$ since $x_j \circledast x_k = u \circledast u = 0$.

\vskip3pt\noindent
\textsf{Case 2.2}: Both $R$ and $R'$ contain at least two leaves.
Let $u$, $v$, $w$, and $z$ be as in \ref{lem:anticommutative:iii}.
By \ref{lem:anticommutative:iv}, there exists an assignment $h \colon \var(t_S) \to U$ such that $h(x_j) = h(x_k) = u$, $h(t_R) = c v$, $h(t_{R'}) = c' w$, and $h(t_S) = cc' (v \circledast w) \in U$ for some nonzero scalars $c$ and $c'$.
If $S=T_t$ then let $h' := h$.
If $S\ne T_t$ then we can use \ref{lem:anticommutative:iv} again to obtain an assignment $h' \colon \var(t) \to U$ extending $h$ such that $h'(t) = c c' c'' z\ne 0$ for some nonzero scalar $c''$.
On the other hand, we have $h'(s) = 0$ since $x_j \circledast x_k = u \circledast u = 0$.
\end{proof}

\subsection{The cross product}

Since the cross product for a three-dimensional Euclidean vector space is anticommutative, we can determine its ac-spectrum by using a commutative operation associated with it.

\begin{definition}
Define $\Join$ on a three-dimensional real vector space $V$ with a basis $\{u,v,w\}$ by letting $x\Join x:=0$ for all $x\in \{u,v,w\}$ and $x\Join y := z$ for all distinct $x,y\in\{u,v,w\}$, where $z \in \{u,v,w\} \setminus \{x,y\}$, and extending this bilinearly from $\{u,v,w\}$ to $V$, i.e.,
\[ (\alpha u + \beta v + \gamma w) \Join (\alpha' u + \beta' v + \gamma' w) = (\beta \gamma' + \gamma \beta') u + (\gamma \alpha' + \alpha \gamma') v + (\alpha \beta' + \beta \alpha') w \]
for any scalars $\alpha$, $\beta$ and $\gamma$.
This operation occurs in recent studies of the Norton algebras of certain distance regular graphs and it is commutative and totally nonassociative~\cite[Example~3.11, Remark~5.10]{Hamming}.
\end{definition}

For $V=\mathbb{R}^3$ with $\{u,v,w\}$ being the standard basis $\{\mathbf{i}, \mathbf{j}, \mathbf{k}\}$, we have a choice function $g$ satisfying $g(\{\pm \mathbf i\})=\mathbf i$, $g(\{\pm \mathbf j\})=\mathbf j$, and $g(\{\pm \mathbf k\})=\mathbf k$, and this makes $\Join$ a commutative version of the cross product $\times$. 

To apply Theorem~\ref{thm:anticommutative}, we need to first verify the assumptions in Lemma~\ref{lem:anticommutative} for $\Join$.

\begin{lemma}\label{lem:Join}
The assumptions \ref{lem:anticommutative:i}--\ref{lem:anticommutative:iv} in Lemma~\ref{lem:anticommutative} hold for $\Join$, $U:=\{u,v,w\}$ and $P:=\{(x,y): x,y\in U, x\ne y\}$.
\end{lemma}

\begin{proof}
It is clear that \ref{lem:anticommutative:i} holds.
For \ref{lem:anticommutative:ii}, we have $(u,v)\in P$ and $(u \Join v, u) = (w,u)\in P$. 
For \ref{lem:anticommutative:iii}, we have $(u,v), (u,w) \in P$ and $(v\Join w, v) = (u,v)\in P$.

It remains to show \ref{lem:anticommutative:iv}, that is, for any $n \geq 2$, $t \in F_n$, $j \in \nset{n}$, and $(a,b)\in P$, there is an assignment of values from $\{u, v, w\}$ to the variables occurring in $t$ such that $x_j$ gets value $a$ and $t$ evaluates to $b$.
We prove this by induction on $n$.
For $n = 2$, this follows from the definition of $\Join$.

Assume now that $n \geq 3$.
Then $t = (rs)$ for some subterms $r$ and $s$.
Assume without loss of generality that $x_j$ occurs in $r$ (the case when $x_j$ occurs in $s$ is treated similarly).
First suppose that $r = x_j$.
By the inductive hypothesis, there is an assignment $h \colon X_n \setminus \{x_j\} \to \{u, v, w\}$ such that $h(s) = a \Join b$.
By further setting $h(x_j) = a$, we get $h(t) = a \Join (a \Join b) = b$.
Now suppose that $r$ contains at least two variables.
By the inductive hypothesis, there is an assignment $h \colon X_n \to \{u, v, w\}$ such that $h(x_j) = a$, $h(r) = a \Join b$, and $h(s) = a$.
Then $h(t) = (a \Join b) \Join a = b$.
\end{proof}

Now we can determine the ac-spectra of the cross product $\times$ and its commutative version $\Join$.

\begin{corollary}\label{cor:cross}
For the cross product $\times$ on $\RR^3$ and its commutative version $\Join$, we have 
\begin{itemize}
\item
$\spac{\Join} = D_{n-1}$ for all $n \geq 1$, 
\item
$\spac{\times} = 2D_{n-1}$ for all $n \geq 2$, and
\item
$\spa{\times} = \spa{\Join} = C_{n-1}$ for all $n\ge1$.
\end{itemize}
\end{corollary}

\begin{proof}
The result follows immediately from Theorem~\ref{thm:anticommutative}, Lemma~\ref{lem:anticommutative} and Lemma~\ref{lem:Join}.
\end{proof}

\subsection{Lie algebras}

A \emph{Lie algebra} is an algebra over a field $\FF$ satisfying the identities
\begin{gather*}
x x \approx 0, \\
x (y z) + y (z x) + z (x y) \approx 0.
\end{gather*}
The first of the above identities implies that the identity $x y \approx - y x$ is also satisfied, since
\[
0 \approx (x+y)(x+y) \approx x x + x y + y x + y y \approx x y + y x.
\]
Hence a Lie algebra over a field of characteristic distinct from $2$ must be anticommutative.
The second one is known as the \emph{Jacobi identity.}
The bilinear product of a Lie algebra is often denoted by $[-,-]$ and is called the \emph{Lie bracket.}
A Lie algebra is \emph{abelian} if its Lie bracket is constantly zero.
Such a Lie algebra is trivially commutative and associative.
In general, a Lie algebra is neither commutative nor associative.

\begin{definition}
A triple $(e,f,h)$ of nonzero elements of a Lie algebra is called an \emph{$\mathfrak{sl}_2$\hyp{}triple} if $[e,f] = h$, $[h,e] = 2e$, and $[h,f] = -2f$.
\end{definition}

It is well known that $\mathfrak{sl}_2$\hyp{}triples exist in every semisimple Lie algebra over a field of characteristic zero.
Consider a Lie algebra over a field of characteristic distinct from $2$ with an $\mathfrak{sl}_2$\hyp{}triple $(e,f,h)$.
One can check that $e,f,h$ must be linearly independent. 
Thus we can fix a basis containing $e, f, h$ and obtain a commutative version $[[- , -]]$ of the Lie bracket $[-,-]$ satisfying 
\begin{equation}\label{eq:triple}
[[e,f]]=[[f,e]]=h, \quad [[h,e]]=[[e,h]]=2e, \qand [[h,f]]=[[f,h]]=2f.
\end{equation}

\begin{lemma}\label{lem:Lie}
Let $\mathbf{L}$ be a Lie algebra over a field of characteristic distinct from $2$ with an $\mathfrak{sl}_2$\hyp{}triple $(e,f,h)$.
Then the conditions \ref{lem:anticommutative:i}--\ref{lem:anticommutative:iv} in Lemma~\ref{lem:anticommutative} hold for the Lie bracket $[-,-]$ and its commutative version $[[-,-]]$ with $U := \{e,f,h\}$ and $P :=  \{(e,e), (f,f), (e,h)\}$.
\end{lemma}

\begin{proof}
Let $\mathbf{[L]} = (L, [[-,-]])$, where $L$ is the universe of $\mathbf{L}$.
It is clear that \ref{lem:anticommutative:i} holds.
For \ref{lem:anticommutative:ii}, we have $(e,h)\in P$ and $(\frac12 [[e,h]], e) = (e,e)\in P$.
For \ref{lem:anticommutative:iii}, we have $(e,e), (e,h) \in P$ and $(\frac12 [[e,h]], e) = (e,e)\in P$.
It remains to show \ref{lem:anticommutative:iv}, i.e., for any $n \geq 2$, $t \in F_n$, $j \in \nset{n}$, and $(u,w)\in P$, there are $u_1, \dots, u_n \in U$ such that $u_j = u$ and $t^{\mathbf{[L]}}(u_1, \ldots, u_n) = cw$ for some scalar $c\ne0$. 
We proceed by induction on $n$.

For $n = 2$, this holds by Equation~\eqref{eq:triple}.
Assume now that $n \geq 3$ and that \ref{lem:anticommutative:iv} holds for bracketings of fewer than $n$ variables.
Then $t = (t_1 t_2)$ for subterms $t_1$ and $t_2$.
Assume $x_j \in \var(t_1)$, without loss of generality.
We distinguish some cases below.

First suppose $(u,w) = (e,e)$.
By applying the inductive hypothesis to $t_1$ and $t_2$ (trivial if $\varcount{t_1} = 1$ or $\varcount{t_2} = 1$), we get elements $u_1, \dots, u_n \in \{e,f,h\}$ such that $u_j = e$, $t_1^\mathbf{[L]}(u_1, \dots, u_\ell) = c_1 e$, and $t_2^\mathbf{[L]}(u_{\ell + 1}, \dots, u_n) = c_2 h$ for some nonzero scalars $c_1$, $c_2$.
Then $t^\mathbf{[L]}(u_1, \dots, u_n) = [[c_1 e, c_2 h]] = 2 c_1 c_2 e$.
A similar argument is valid for the case $(u,w) = (f,f)$.

Now suppose $(u,w) = (e,h)$.
By the inductive hypothesis (trivial if $\varcount{t_1} = 1$ or $\varcount{t_2} = 1$), there exist $u_1, \dots, u_n \in \{e,f,h\}$ such that $u_j = e$, $t_1^\mathbf{[L]}(u_0, \dots, u_\ell) = c_1 e$, and $t_2^\mathbf{[L]}(u_{\ell + 1}, \dots, u_n) = c_2 f$ for some nonzero scalars $c_1$, $c_2$.
Then $t^\mathbf{[L]}(u_1, \dots, u_n) = [c_1 e, c_2 f] = c_1 c_2 h$.
\end{proof}

\begin{corollary}\label{cor:Lie}
Let $\mathbf{L}$ be a Lie algebra over a field of characteristic distinct from $2$ with an $\mathfrak{sl}_2$-triple.
For the Lie bracket $[-,-]$ of $\mathbf{L}$, it holds that $\spac{[-,-]} = 2D_{n-1}$ for all $n \geq 2$ and $\spa{[-,-]} = C_{n-1}$ for all $n \geq 1$.
\end{corollary}

\begin{proof}
The result follows immediately from Theorem~\ref{thm:anticommutative}, Lemma~\ref{lem:anticommutative} and Lemma~\ref{lem:Lie}.
\end{proof}

\section{Totally nonassociative operations}\label{sec:TN}

In this section we focus on the ac\hyp{}spectra of some totally nonassociative operations that are not commutative or anticommutative.
Recall that a binary operation $*$ is said to be \emph{totally nonassociative} if $\spa{*} = C_{n-1}$ for all $n \geq 1$.
The arithmetic, geometric, and harmonic means, the cross product on $\RR^3$, and the Lie brackets of Lie algebras over fields of characteristic distinct from $2$ with an $\mathfrak{sl}_2$-triple are all totally nonassociative and their ac-spectra have been determined in earlier sections.
There are many other examples of totally nonassociative operations~\cite{AssociativeSpectra1,Hamming}.
We will study the exponentiation, the implication, and the negated disjunction (NOR) in this section.

\subsection{Exponentiation}
Recall that any binary operation $*$ satisfies $\spac{*} \leq n! C_{n-1}$ for all $n \geq 1$, and the equality is achieved by the free groupoid on one generator (Proposition \ref{prop:free-groupoid}).
We show that if $*$ is the exponentiation then $\spac{*}$ is strictly less than this upper bound.

We need the following lemma, which applies to the exponentiation.

\begin{lemma}[{Cs\'ak\'any, Waldhauser \cite[statements 4.2.1 and 4.2.2]{AssociativeSpectra1}}]
\label{lem:surjective}
Let $\mathbf{G} = (G, \ast)$ be a groupoid, and assume that the operation $\ast:G\times G\to G$ is surjective.
Let $t \in T(X_\omega)$ be a linear term.
Then the following statements hold.
\begin{enumerate}[label={\upshape(\roman*)}]
\item The term operation $t^\mathbf{G}$ is surjective.
\item If $\ast$ is surjective and its Cayley table has neither two identical columns nor two identical rows, then $t^\mathbf{G}$ depends on every variable in $\var(t)$, 
i.e., for every $x\in\var(t)$, there exist assignments $h,h':\var(t)\to G$ such that $h(z)=h'(z)$ for all $z\in \var(t)\setminus\{x\}$ and $h(t)\ne h'(t)$.
\end{enumerate}
\end{lemma}

\begin{definition}
Let $t$ be a term.
If $t$ is not a variable, then it is of the form $t = (t_\mathrm{L} t_\mathrm{R})$ for some subterms $t_\mathrm{L}$ and $t_\mathrm{R}$.
If the left subterm $t_\mathrm{L}$ is not a variable, we can further write it in the form $t_\mathrm{L} = (t'_\mathrm{L} t'_\mathrm{R})$.
Continuing in this way, always further decomposing left subterms until we reach one that is a variable, we can write the term $t$ in the form
$t = ( ( \cdots ( ( x_i t_1 ) t_2 ) \cdots ) t_k )$,
which we will refer to as the \emph{leftmost decomposition} of $t$, and the subterms $t_1, \dots, t_k$ are called the \emph{factors} of the decomposition.
The variable $x_i$ is called the \emph{leftmost variable} of $t$.
Denote the leftmost variable of $t$ by $L(t)$.
\end{definition}

Note that the number of factors in a leftmost decomposition of $t$ equals the number of opening parentheses preceding the leftmost variable in $t$.
We can think of this in terms of the binary tree $T_t$.
The factors of the leftmost decomposition correspond to the subtrees $T_{t_1}, \dots, T_{t_k}$ rooted at the right children of the internal vertices along the unique path from the leftmost leaf to the root of $T_t$.

\begin{definition}
Based on leftmost decomposition, we associate with each term $t \in T(X)$ an ordered labeled tree $P_t$ that is defined by the following recursion.
If $t$ is a variable $x_i \in X$, then $P_t$ is the one\hyp{}vertex tree with the single vertex labeled by $x_i$.
If $t$ is not a variable and $t = ( ( \cdots ( ( x_i t_1 ) t_2 ) \cdots ) t_k )$ is the leftmost decomposition of $t$, then $P_t$ is the ordered tree whose root is labeled by $x_i$ and has $k$ children at which the subtrees $P_{t_1}, P_{t_2}, \dots, P_{t_k}$ are rooted.
\end{definition}

\begin{proposition}\label{prop:exp}
Let $\mathbf{G} = (G,*)$ be a groupoid satisfying the identity $(xy)z \approx (xz)y$.
If $s, t \in T(X_\omega)$ are linear terms such that the corresponding ordered labeled trees $P_s$ and $P_t$ have equal underlying unordered trees, i.e., $P_s^{\mathrm{u}} =  P_t^{\mathrm{u}}$, then $s^{\mathbf{G}} = t^{\mathbf{G}}$.
Consequently, $\spac{\mathbf{G}} \le n^{n-1}$.
Moreover, if the equality holds, then $\spa{\mathbf{G}}=C_{n-1}$.
\end{proposition}

\begin{proof}
Since $\mathbf{G}$ satisfies the identity $(xy)z \approx (xz)y$, it also satisfies the identity
\begin{equation}
( ( \cdots ( ( x_0 x_1 ) x_2 ) \cdots ) x_n )
\approx
( ( \cdots ( ( x_0 x_{\sigma(1)} ) x_{\sigma(2)} ) \cdots ) x_{\sigma(n)} )
\label{eq:exp-id}
\end{equation}
for every $n \in \IN_{+}$ and $\sigma \in \SS_n$.
In other words, permuting the factors of the leftmost decomposition of a term does not alter the induced term operation on $\mathbf{G}$.

Suppose $s, t \in T(X_\omega)$ with $P_s^{\mathrm{u}} = P_t^{\mathrm{u}}$.
We prove $s^{\mathbf{G}} = t^{\mathbf{G}}$ by induction on $n := \varcount{s} = \varcount{t}$.
This clearly holds for $1 \leq n \leq 2$.
Assume now that $n \geq 3$ and that the claim holds for terms with fewer than $n$ variables.
Let $s = ( ( \cdots ( ( x_i s_1 ) s_2 ) \cdots ) s_d )$ and $t = ( ( \cdots ( ( x_j t_1 ) t_2 ) \cdots ) t_e )$
be the leftmost decompositions of $s$ and $t$.

Since $P_s^{\mathrm{u}} = P_t^{\mathrm{u}}$, the roots of $P_s$ and $P_t$ are labeled with the same variable and have the same number of children, that is, $x_i = x_j$ and $d = e$.
The subtrees rooted at the children of the root of $P_s$ are $P_{s_1}, \dots, P_{s_d}$, and, similarly
the subtrees rooted at the children of the root of $P_t$ are $P_{t_1}, \dots, P_{t_d}$.
Since $P_s^{\mathrm{u}} = P_t^{\mathrm{u}}$, there is a permutation $\pi \in \SS_d$ such that $P_{s_i}^{\mathrm{u}} = P_{t_{\pi(i)}}^{\mathrm{u}}$ for all $i \in \nset{d}$.
By the induction hypothesis, we have $s_i^\mathbf{G} = t_{\pi(i)}^\mathbf{G}$ for all $i \in \nset{d}$.
Consequently,
\[
\begin{split}
s^\mathbf{G}(\mathbf{a})
&
= ( ( \cdots ( ( x_i^\mathbf{G}(\mathbf{a}) * s_1^\mathbf{G}(\mathbf{a}) ) * s_2^\mathbf{G}(\mathbf{a}) ) * \cdots ) * s_d^\mathbf{G}(\mathbf{a}) )
\\ &
= ( ( \cdots ( ( x_i^\mathbf{G}(\mathbf{a}) * t_{\pi(1)}^\mathbf{G}(\mathbf{a}) ) * t_{\pi(2)}^\mathbf{G}(\mathbf{a}) ) * \cdots ) * t_{\pi(d)}^\mathbf{G}(\mathbf{a}) )
\\ &
= ( ( \cdots ( ( x_i^\mathbf{G}(\mathbf{a}) * t_1^\mathbf{G}(\mathbf{a}) ) * t_2^\mathbf{G}(\mathbf{a}) ) * \cdots ) * t_d^\mathbf{G}(\mathbf{a}) )
= t^\mathbf{G}(\mathbf{a})
\end{split}
\]
for all $\mathbf{a}$ in the domain of $s^\mathbf{G}$, so $s^\mathbf{G} = t^\mathbf{G}$.
It follows that $\spac{\mathbf{G}}$ is bounded above by the number of unordered rooted trees with $n$ labeled vertices, which is $n^{n-1}$ (see, e.g., Tak\'acs~\cite[\S3]{Takacs}).

Now suppose $\spac{\mathbf{G}} = n^{n-1}$, i.e., $P_s^{\mathrm{u}} = P_t^{\mathrm{u}}$ if and only if $s^{\mathbf{G}} = t^{\mathbf{G}}$ for all $s, t \in T(X_\omega)$.
If two distinct binary trees have leaves labeled $1,\ldots,n$ from left to right, then they correspond to distinct ordered trees with nodes labeled $1, \ldots, n$ in the pre-order (first visit the root and then recursively visit the subtrees rooted at the children of the root from left to right), so the underlying unordered labeled trees are distinct and they induce distinct term operations on $\mathbf{G}$.
This shows that $\spa{\mathbf{G}} = C_{n-1}$.
\end{proof}

\begin{proposition}
\label{prop:exponentiation}
For $\mathbf{G} = (\RR_{\geq 0}, {*})$, where $*$ is the exponentiation operation defined by $a * b := a^b$ for all $a, b \in \RR_{\geq 0}$, we have $\spac{\mathbf{G}} = n^{n-1}$ and $\spa{\mathbf{G}} = C_{n-1}$.
\end{proposition}

\begin{proof} 
Proposition~\ref{prop:exp} applies to the groupoid $\mathbf{G} = (\RR_{\geq 0}, {*})$ since it satisfies $(xy)z \approx (xz)y$ by the power rule of exponents: $(a^b)^c = a^{bc} = (a^c)^b$.
Thus it suffices to show that $s^\mathbf{G} = t^\mathbf{G}$ implies $P_s^{\mathrm{u}} = P_t^{\mathrm{u}}$ for any linear terms $s, t \in T(X_\omega)$.

Assume $s^\mathbf{G} = t^\mathbf{G}$, which implies $\var(s) = \var(t)$ by Lemma~\ref{lem:surjective}.
We show $P_s^{\mathrm{u}} = P_t^{\mathrm{u}}$ by induction on $n := \varcount{s} = \varcount{t}$.
This is trivial for $1 \leq n \leq 2$ and we can assume it holds for linear terms with fewer than $n$ variables, where $n\ge3$.
Let $s = ( ( \cdots ( ( x_i s_1 ) s_2 ) \cdots ) s_d )$ and $t = ( ( \cdots ( ( x_j t_1 ) t_2 ) \cdots ) t_e )$
be the leftmost decompositions of $s$ and $t$.

Let $h \colon X_\omega \to \RR_{\geq 0}$ be an assignment of distinct prime numbers to the variables, e.g., $x_i \mapsto p_i$, where $p_i$ is the $i$\hyp{}th prime, for each $i \in \IN_{+}$.
We have
\begin{align*}
h(s)
&
= ( ( \cdots ( ( h(x_i) * h(s_1) ) * h(s_2) ) * \cdots ) * h(s_d) )
= p_i^{h(s_1) h(s_2) \cdots h(s_d)},
\\
h(t)
&
= ( ( \cdots ( ( h(x_j) * h(t_1) ) * h(t_2) ) * \cdots ) * h(t_e) )
= p_j^{h(t_1) h(t_2) \cdots h(t_e)}. 
\end{align*}
Since $s^\mathbf{G} = t^\mathbf{G}$, we must have $h(s) = h(t)$.
It follows from the fundamental theorem of arithmetic that $p_i = p_j$ and $h(s_1) h(s_2) \cdots h(s_d) = h(t_1) h(t_2) \cdots h(t_e)$.
Therefore $x_i = x_j$.
Moreover, since $s$ is a linear term, the sets $\var(s_1), \var(s_2), \dots, \var(s_d)$ are pairwise disjoint and so $h(s_1), h(s_2), \dots, h(s_d)$ are powers of distinct primes; similarly $\var(t_1), \var(t_2), \dots, \var(t_e)$ are pairwise disjoint and so $h(t_1), h(t_2), \dots, h(t_e)$ are powers of distinct primes.
It follows again from the fundamental theorem of arithmetic that $d = e$ and there is a unique permutation $\pi \in \SS_d$ such that $h(s_\ell) = h(t_{\pi(\ell)})$ for all $\ell \in \nset{d}$.
Now, by considering, for each $\ell \in \nset{d}$, all assignments that fix $x_i \mapsto 2$, $x_p \mapsto 1$ for all $x_p \in X_\omega \setminus (\var(s_\ell) \cup \{x_i\})$ and let the values for the variables in $\var(s_\ell)$ vary, we see that $s_\ell^\mathbf{G} = t_{\pi(\ell)}^\mathbf{G}$ for all $\ell \in \nset{d}$.
By the induction hypothesis, we have $P_{s_\ell}^{\mathrm{u}} = P_{s_{\pi(\ell)}}^{\mathrm{u}}$, and it follows that $P_s^{\mathrm{u}} = P_t^{\mathrm{u}}$.
\end{proof}

\subsection{Implication}
Now we study the implication $\rightarrow$ defined on the set $\{0,1\}$ by $x \rightarrow y =0$ if $x=1$ and $y=0$ or $x \rightarrow y =0$ otherwise.
It turns out to be more convenient to use the converse implication $\leftarrow$, which is defined on $\{0,1\}$ by $x \leftarrow y := y \rightarrow x$.
As the groupoids $(\{0,1\},{\rightarrow})$ and $(\{0,1\},{\leftarrow})$ are antiisomorphic, their ac-spectra coincide;
therefore it suffices to consider only $\leftarrow$.

Note that the implication $\rightarrow$ and the converse implication $\leftarrow$ are surjective operations on $\{0,1\}$, and their Cayley tables have no two identical columns nor two identical rows.
Hence Lemma~\ref{lem:surjective} applies to $\rightarrow$ and $\leftarrow$.

\begin{proposition}\label{prop:implication}
For $\mathbf{G} = (\{0,1\}, {*})$, where $*$ is the converse implication $\leftarrow$, we have $\spac{\mathbf{G}} = n^{n-1}$ and $\spa{\mathbf{G}} = C_{n-1}$.
\end{proposition}

\begin{proof}
Observe first that $\mathbf{G}$ satisfies the identity $(x_1 x_2) x_3 \approx (x_1 x_3) x_2$.
Therefore Proposition~\ref{prop:exp} applies to $\mathbf{G}$. 
It remains to show that $s^\mathbf{G} = t^\mathbf{G}$ implies $P_s^{\mathrm{u}} = P_t^{\mathrm{u}}$ for any linear terms $s, t \in T(X_\omega)$.

Assume $s^\mathbf{G} = t^\mathbf{G}$, which implies $\var(s) = \var(t)$ by Lemma~\ref{lem:surjective}.
We show $P_s^{\mathrm{u}} = P_t^{\mathrm{u}}$ by induction on $n := \varcount{s} = \varcount{t}$.
This is trivial for $1 \leq n \leq 2$ and we can assume it holds for linear terms with fewer than $n$ variables, where $n\ge3$.

Let $h \colon X_\omega \to \{0,1\}$ be the assignment with $h(x_i) = 0$ and $h(x_k) = 1$ for $k \neq i$.
Then clearly
\begin{align*}
h(s) &= ( ( \cdots ( (h(x_i) \ast h(s_1)) \ast h(s_2)) \ast \cdots ) \ast h(s_d))
= ( ( \cdots ( ( 0 \ast 1 ) \ast 1 ) \ast \cdots ) \ast 1 )
= 0,
\\
h(t) &= ( ( \cdots ( (h(x_j) \ast h(t_1)) \ast h(t_2)) \ast \cdots ) \ast h(t_e))
=
\begin{cases}
0, & \text{if $i = j$,} \\
1, & \text{if $i \neq j$.}
\end{cases}
\end{align*}
Since $s^\mathbf{G} = t^\mathbf{G}$, we must have $i = j$.

We claim that for every $k \in \nset{d}$, there exists some $\ell \in \nset{e}$ such that $L(s_k) = L(t_\ell)$.
Suppose, to the contrary, that $L(s_k) =: x_\alpha$ is not the leftmost variable of any factor $t_\ell$.
Consider the assignment $h$ that maps $x_i$ and $x_\alpha$ to $0$ and all remaining variables to $1$.
We now have that $h(s_k) = 0$, $h(s_m) = 1$ for any $m \neq k$, and $h(t_1) = \dots = h(t_d) = 1$. This leads to a contradiction:
\begin{align*}
h(s)
&=( ( \cdots ( ( ( ( \cdots ( ( 0 \ast 1 ) \ast 1 ) \ast \cdots ) \ast 1 ) \ast 0 ) \ast 1 ) \ast \cdots ) \ast 1)
= 1,
\\
h(t)
&= ( ( \cdots ( (0 \ast 1) \ast 1) \ast \cdots ) \ast 1)
= 0.
\end{align*}
A similar argument shows that for every $\ell \in \nset{e}$, we have $L(t_\ell) = L(s_k)$ for some $k \in \nset{d}$.
Consequently, $d = e$ and there exists a bijection $\pi \in \SS_d$ such that $L(s_k) = L(t_{\pi(k)})$ for all $k \in \nset{d}$.

Let $\ell \in \nset{d}$, and
suppose $h \colon X_\omega \to \{0,1\}$ is an assignment satisfying $x_i \mapsto 0$ and $L(s_k) \mapsto 1$ for all $k \in \nset{d} \setminus \{\ell\}$.
It is easy to see that then $h(s) = 0 \ast h(s_\ell) = \overline{h(s_\ell)}$, where we use the notation $\overline0:=1$ and $\overline 1:=0$.
Similarly, if $h$ satisfies $x_i \mapsto 0$ and $L(t_k) \mapsto 1$ for all $k \in \nset{d} \setminus \{\ell\}$, then $h(t) = 0 \ast h(t_\ell) = \overline{h(t_\ell)}$.

Considering the kind of assignments described above, we can conclude that $\var(s_\ell) = \var(t_{\pi(\ell)})$ for all $\ell \in \nset{d}$.
(Suppose, to the contrary, that there is a variable $z \in \var(s_\ell) \setminus \var(t_{\pi(\ell)})$.
By Lemma~\ref{lem:surjective}, $s_\ell^\mathbf{G}$ depends on every variable in $\var(s_\ell)$, in particular on $z$; therefore there exist assignments $h, h' \colon \var(s_\ell) \to \{0,1\}$ such that $h(x) = h'(x)$ whenever $x \neq z$ and $h(s_\ell) \neq h'(s_\ell)$.
Extend both $h$ and $h'$ to $X_\omega$ by mapping $x_i \mapsto 0$ and $x_p \mapsto 1$ for all $x_p \in X_\omega \setminus (\var(s_\ell) \cup \{x_i\})$.
Now $h(s) = \overline{h(s_\ell)} \neq \overline{h'(s_\ell)} = h'(s)$.
On the other hand, since $h$ and $h'$ only disagree at $z$ and $z \notin \var(t_{\pi(\ell)})$, it follows that $h(t) = \overline{h(t_{\pi(\ell)})} = \overline{h'(t_{\pi(\ell)})} = h'(t)$.
Consequently, $s^\mathbf{G} \neq t^\mathbf{G}$, a contradiction.
Thus $\var(s_\ell) \subseteq \var(t_{\pi(\ell)})$.
A similar argument shows that $\var(t_{\pi(\ell)}) \subseteq \var(s_\ell)$.)
Now, for all assignments $h$ with $h(x_i) = 0$ and 
$h(L(s_k)) = 1$ for all $k \in \nset{d} \setminus \{\ell\}$
we have $h(s_\ell) = \overline{h(s)} = \overline{h(t)} = h(s_{\pi(\ell)})$, so $s_\ell^\mathbf{G} = t_{\pi(\ell)}^\mathbf{G}$ for all $\ell \in \nset{d}$.
By the induction hypothesis, $P_{s_\ell}^{\mathrm{u}} = P_{t_{\pi(\ell)}}^{\mathrm{u}}$ for all $\ell \in \nset{d}$, and, consequently, $P_s^\mathrm{u} = P_t^\mathrm{u}$.
\end{proof}

\subsection{Negated disjunction (NOR)}
Now we study the ac-spectrum of the groupoid $\mathbf{G} = (\{0,1\}, {\NOR})$, where $\NOR$ is the \emph{negated disjunction} (\emph{NOR}), defined by the rule $x \NOR y = 1$ if and only if $x = y = 0$.
Note that the negated disjunction is surjective and its Cayley table has no two identical columns nor two identical rows.
Hence Lemma~\ref{lem:surjective} applies thereto.

\begin{definition}
\leavevmode	
\begin{enumerate}[label={\upshape(\roman*)}]
\item If $f_i \colon A_i \to B$ ($i \in I$) are functions with pairwise disjoint domains $A_i$, then we can define their \emph{union} $\bigcup_{i \in I} f_i \colon \bigcup_{i \in I} A_i \to B$ by the condition that for all $i \in I$, the restriction of $\bigcup_{i \in I} f_i$ to $A_i$ coincides with $f_i$.
In the case when the index set $I$ is finite, we may use notation such as $f \cup g$ or $f \cup g \cup h$.
\item
Let $\mathbf{G} = (G,{\circ})$ be a groupoid.
Let $t \in T(X_\omega)$ be a linear term.
If $\varcount{t} \geq 2$, then $t = (t_1 t_2)$.
Since $t$ is linear, the sets $\var(t_1)$ and $\var(t_2)$ are disjoint,
and therefore we can write any assignment $h \colon \var(t) \to G$ in a unique way as the union of the two assignments $h_1 \colon \var(t_1) \to G$ and $h_2 \colon \var(t_2) \to G$ that are the restrictions of $h$ to $\var(t_1)$ and $\var(t_2)$, respectively.
Then it clearly holds that $h = h_1 \cup h_2$ and $h(t) = (h_1 \cup h_2)(t) = h_1(t_1) \circ h_2(t_2)$.
\end{enumerate}
\end{definition}

Let $\mathbf{G} = (\{0,1\}, {\NOR})$. 
For a term $t \in T(X_\omega)$, let $\mathsf{T}_t$ and $\mathsf{F}_t$ be the sets of \emph{true} and \emph{false assignments} of $t$, that is,
\begin{align*}
\mathsf{T}_t &:= \{ h \colon \var(t) \to \{0,1\} \mid h(t) = 1 \}, \\
\mathsf{F}_t &:= \{ h \colon \var(t) \to \{0,1\} \mid h(t) = 0 \}.
\end{align*}
If $t$ is a linear term with $\varcount{t} \geq 2$, then $t = (t_1 t_2)$, $\var(t_1) \cap \var(t_2) = \emptyset$, and we have
\begin{align*}
\mathsf{T}_t &= \{ h_1 \cup h_2 \mid h_1 \in \mathsf{F}_{t_1} \wedge h_2 \in \mathsf{F}_{t_2} \}, \\
\mathsf{F}_t &= \{ h_1 \cup h_2 \mid h_1 \in \mathsf{T}_{t_1} \vee h_2 \in \mathsf{T}_{t_2} \}.
\end{align*}

\begin{lemma}\label{lem:aux-prod}
Let $k \in \{2, 3, 4\}$, let $S_1$, \dots, $S_k$ be nonempty sets, and let $\mathbf{S} := S_1 \times \dots \times S_k$.
\begin{enumerate}[label={\upshape(\roman*)}]
\item\label{lem:aux-prod:2}
If $k = 2$, let
$A \subseteq S_1$, $B \subseteq S_2$, $C \subseteq S_1$, $D \subseteq S_2$,
and let
$U := A \times B$, $V := C \times D$.

\item\label{lem:aux-prod:3}
If $k = 3$, let
$A \subseteq S_1 \times S_2$, $B \subseteq S_3$, $C \subseteq S_1$, $D \subseteq S_2 \times S_3$, and let
\begin{align*}
U := \{(a_1, a_2, a_3) \in \mathbf{S} \mid (a_1, a_2) \in A, a_3 \in B\}, \\
V := \{(a_1, a_2, a_3) \in \mathbf{S} \mid a_1 \in C, (a_2, a_3) \in D\}.
\end{align*}

\item\label{lem:aux-prod:4}
If $k = 4$, let
$A \subseteq S_1 \times S_2$, $B \subseteq S_3 \times S_4$, $C \subseteq S_1 \times S_3$, $D \subseteq S_2 \times S_4$, and let
\begin{align*}
U := \{(a_1, a_2, a_3, a_4) \in \mathbf{S} \mid (a_1, a_2) \in A, (a_3, a_4) \in B\}, \\
V := \{(a_1, a_2, a_3, a_4) \in \mathbf{S} \mid (a_1, a_3) \in C, (a_2, a_4) \in D\}.
\end{align*}
\end{enumerate}
Assume that $U = \overline{V} := \mathbf{S} \setminus V$.
Then there exists a $j \in \nset{k}$ such that
for all tuples $\mathbf{a} = (a_1, \dots, a_k)$ and $\mathbf{b} = (b_1, \dots, b_k)$ in $\mathbf{S}$ satisfying $a_i = b_i$ for all $i \neq j$, we have $\{\mathbf{a}, \mathbf{b}\} \subseteq U$ or $\{\mathbf{a}, \mathbf{b}\} \subseteq V$.
\end{lemma}

\begin{proof}
Suppose, to the contrary, that no such $j$ exists.
The argument is slightly different for the different values of $k$, and we consider separately the three different cases.

\ref{lem:aux-prod:2}
Assume $k = 2$.
There exist $p_i, q_i \in S_i$ ($i \in \nset{2}$) and
$p'_1 \in S_1$, $q'_2 \in S_2$ such that
$(p_1, p_2), (q_1, q_2) \in U$,
$(p'_1, p_2), (q_1, q'_2) \in V$.
By the definition of the sets $U$ and $V$, we have
$p_1, q_1 \in A$,
$p_2, q_2 \in B$,
$p'_1, q_1 \in C$,
$p_2, q'_2 \in D$.

Since $(q_1, q_2) \in U = \overline{V}$ and $q_1 \in C$, we must have $q_2 \notin D$.
It follows that $(z_1, q_2) \in \overline{V} = U$ for all $z_1 \in S_1$.
Since $(p'_1, p_2) \in V = \overline{U}$ and $p_2 \in B$, we must have $p'_1 \notin A$.
It follows that $(p'_1, z_2) \in \overline{U} = V$ for all $z_2 \in S_2$.
Therefore $(p'_1, q_2) \in U \cap V$, which is a contradiction to our assumption that $U = \overline{V}$.

\ref{lem:aux-prod:3}
Assume $k = 3$.
There exist $p_i, q_i, r_i \in S_i$ ($i \in \nset{3}$) and
$p'_1 \in S_1$, $q'_2 \in S_2$, $r'_3 \in S_3$ such that
\begin{gather*}
(p_1, p_2, p_3), (q_1, q_2, q_3), (r_1, r_2, r_3) \in U,
\\
(p'_1, p_2, p_3), (q_1, q'_2, q_3), (r_1, r_2, r'_3) \in V.
\end{gather*}
By the definition of the sets $U$ and $V$, we have
\begin{align*}
& (p_1, p_2), (q_1, q_2), (r_1, r_2) \in A, &
& p_3, q_3, r_3 \in B, \\
& p'_1, q_1, r_1 \in C, &
& (p_2, p_3), (q'_2, q_3), (r_3, r'_3) \in D.
\end{align*}

Since $(p_1, p_2, p_3) \in U = \overline{V}$ and $(p_2, p_3) \in D$, we must have $p_1 \notin C$.
It follows that $(p_1, z_2, z_3) \in \overline{V} = U$ for all $z_2 \in S_2$ and $z_3 \in S_3$, and,
in particular,
$z_3 \in B$ for all $z_3 \in S_3$, i.e., $B = S_3$.
On the other hand, since $(r_1, r_2, r'_3) \in V = \overline{U}$ and $(r_1, r_2) \in A$, we must have $r'_3 \notin B = S_3$.
We have reached a contradiction.

\ref{lem:aux-prod:4}
Assume $k = 4$.
There exist $p_i, q_i, r_i, s_i \in S_i$ for all $i \in \nset{4}$ and
$p'_1 \in S_1$, $q'_2 \in S_2$, $r'_3 \in S_3$, $s'_4 \in S_4$ such that
\begin{gather*}
(p_1, p_2, p_3, p_4), (q_1, q_2, q_3, q_4), (r_1, r_2, r_3, r_4), (s_1, s_2, s_3, s_4) \in U,
\\
(p'_1, p_2, p_3, p_4), (q_1, q'_2, q_3, q_4), (r_1, r_2, r'_3, r_4), (s_1, s_2, s_3, s'_4) \in V.
\end{gather*}
By the definition of the sets $U$ and $V$, we have
\begin{align*}
& (p_1, p_2), (q_1, q_2), (r_1, r_2), (s_1, s_2) \in A, &
& (p_3, p_4), (q_3, q_4), (r_3, r_4), (s_3, s_4) \in B, \\
& (p'_1, p_3), (q_1, q_3), (r_1, r'_3), (s_1, s_3) \in C, &
& (p_2, p_4), (q'_2, q_4), (r_2, r_4), (s_2, s'_4) \in D.
\end{align*}

Since $(p_1, p_2, p_3, p_4) \in U = \overline{V}$ and $(p_2, p_4) \in D$, we must have $(p_1, p_3) \notin C$.
It follows that $(p_1, z_2, p_3, z_4) \in \overline{V} = U$ for all $z_2 \in S_2$ and $z_4 \in S_4$, and,
in particular, $(p_1, z_2) \in A$ for all $z_2 \in S_2$.

Since $(s_1, s_2, s_3, s_4) \in U = \overline{V}$ and $(s_1, s_3) \in C$, we must have $(s_2, s_4) \notin D$.
It follows that $(z_1, s_2, z_3, s_4) \in \overline{V} = U$ for all $z_1 \in S_1$ and $z_3 \in S_3$, and,
in particular, $(z_3, s_4) \in B$ for all $z_3 \in S_3$.

Since $(r_1, r_2, r'_3, r_4) \in V = \overline{U}$ and $(r_1, r_2) \in A$, we must have $(r'_3, r_4) \notin B$.
It follows that $(z_1, z_2, r'_3, r_4) \in \overline{U} = V$ for all $z_1 \in S_1$ and $z_2 \in S_2$, and,
in particular, $(z_1, r'_3) \in C$ for all $z_1 \in S_1$.

Since $(q_1, q'_2, q_3, q_4) \in V = \overline{U}$ and $(q_3, q_4) \in B$, we must have $(q_1, q'_2) \notin A$.
It follows that $(q_1, q'_2, z_3, z_4) \in \overline{U} = V$ for all $z_3 \in S_3$ and $z_4 \in S_4$, and,
in particular, $(q'_2, z_4) \in D$ for all $z_4 \in S_4$.

Therefore
$(p_1, q'_2) \in A$, $(r'_3, s_4) \in B$, $(p_1, r'_3) \in C$, and $(q'_2, s_4) \in D$, so
$(p_1, q'_2, r'_3, s_4) \in U \cap V$.
This is a contradiction to our assumption that $U = \overline{V}$.
\end{proof}

\begin{proposition}\label{prop:Sheffer}
For $\mathbf{G} = (\{0,1\}, {\NOR})$, we have $\spac{\mathbf{G}} = D_{n-1}$.
\end{proposition}

\begin{proof}
Since $\NOR$ is commutative, we have $\spac{\mathbf{G}} \leq D_{n-1}$.
In order to prove that the equality holds, we need to show that for all linear terms $s, t \in T(X_\omega)$, $s^\mathbf{G} = t^\mathbf{G}$ implies that the leaf\hyp{}labeled unordered binary trees $T_s^\mathrm{u}$ and $T_t^\mathrm{u}$ are isomorphic.

Assume $s^\mathbf{G} = t^\mathbf{G}$, which implies $\var(s) = \var(t) =: Y$ by Lemma~\ref{lem:surjective}.
We proceed by induction on $n := \varcount{s} = \varcount{t}$.
The claim is obvious for $n \leq 2$.
Assume now that $n \geq 3$ and that the claim holds for linear terms with fewer than $n$ variables.
Then $s = (s_1 s_2)$ and $t = (t_1 t_2)$.

First consider the case when $\{\var(s_1), \var(s_2)\} = \{\var(t_1), \var(t_2)\}$; by the commutativity of $\NOR$, we may assume that $\var(s_1) = \var(t_1)$ and $\var(s_2) = \var(t_2)$.
Suppose, to the contrary, that $T_s^\mathrm{u}$ and $T_t^\mathrm{u}$ are not isomorphic.
Then $T_{s_1}^\mathrm{u} \not\cong T_{t_1}^\mathrm{u}$ or $T_{s_2}^\mathrm{u} \not\cong T_{t_2}^\mathrm{u}$; without loss of generality, assume that $T_{s_1}^\mathrm{u} \not\cong T_{t_1}^\mathrm{u}$.
By the induction hypothesis, $s_1^\mathbf{G} \neq t_1^\mathbf{G}$, so there exists an assignment $h_1 \colon \var(s_1) \to \{0,1\}$ such that $h_1(s_1) \neq h_1(t_1)$; without loss of generality, assume $h_1(s_1) = 0$ and $h_1(t_1) = 1$.
By Lemma~\ref{lem:surjective}, $s_2^\mathbf{G}$ is surjective, so there is an assignment $h_2 \colon \var(s_2) \to \{0,1\}$ such that $h_2(s_2) = 0$.
Now $(h_1 \cup h_2)(s) = h_1(s_1) \NOR h_2(s_2) = 0 \NOR 0 = 1$ and $(h_1 \cup h_2)(t) = h_1(t_1) \NOR h_2(t_2) = 1 \NOR h_2(t_2) = 0$, so $s^\mathbf{G} \neq t^\mathbf{G}$, a contradiction.

Suppose now that $\{\var(s_1), \var(s_2)\} \neq \{\var(t_1), \var(t_2)\}$.
There are pairwise disjoint sets $Y_1, Y_2, Y_3, Y_4 \subseteq Y$ such that
$\var(s_1) = Y_1 \cup Y_2$,
$\var(s_2) = Y_3 \cup Y_4$,
$\var(t_1) = Y_1 \cup Y_3$,
$\var(t_2) = Y_2 \cup Y_4$,
and at most one of the sets $Y_i$ is empty; without loss of generality, assume that only $Y_4$ is potentially empty.
Let
\[
W :=
\{(h_1, h_2, h_3, h_4) \in \{0,1\}^{Y_1} \times \{0,1\}^{Y_2} \times \{0,1\}^{Y_3} \times \{0,1\}^{Y_4} \mid h_1 \cup h_2 \cup h_3 \cup h_4 \in \mathsf{T}_s = \mathsf{T}_t \}.
\]
Note that for all $(h_1, h_2, h_3, h_4) \in W$, we have $(h_1 \cup h_2)(s_1) = (h_3 \cup h_4)(s_2) = (h_1 \cup h_3)(t_1) = (h_2 \cup h_4)(t_2) = 0$.

\vskip5pt \noindent
\textsf{Claim:} There exist $A \subseteq \{0,1\}^{Y_1}$, $B \subseteq \{0,1\}^{Y_2}$, $C \subseteq \{0,1\}^{Y_3}$, $D \subseteq \{0,1\}^{Y_4}$ such that $W = A \times B \times C \times D$.

\vskip5pt\noindent
\textsf{Proof of the Claim:}
Let $A$, $B$, $C$, and $D$ be the projections of $W$ into the first, second, third, and fourth components, respectively, i.e.,
$A := \{ h_1 \in \{0,1\}^{Y_1} \mid (h_1, h_2, h_3, h_4) \in W\}$ and similarly for $B$, $C$, $D$.
We clearly have $W \subseteq A \times B \times C \times D$.
In order to prove the reverse inclusion, let $(a, b, c, d) \in A \times B \times C \times D$.
Then there exist $(a_i, b_i, c_i, d_i) \in W$ for all $i \in \nset{4}$ such that $a = a_1$, $b = b_2$, $c = c_3$, $d = d_4$.
Now $(a_1 \cup c_1)(t_1) = 0$ and $(b_2 \cup d_2)(t_2) = 0$, so $(a_1 \cup b_2 \cup c_1 \cup d_2)(t) = 1$, i.e., $(a_1, b_2, c_1, d_2) \in W$.
Similarly, $(a_3 \cup c_3)(t_1) = 0$ and $(b_4 \cup d_4)(t_2) = 0$, so $(a_3 \cup b_4 \cup c_3 \cup d_4)(t) = 1$, i.e., $(a_3, b_4, c_3, d_4) \in W$.
It follows that $(a_1 \cup b_2)(s_1) = 0$ and $(c_3 \cup d_4)(s_2) = 0$, so $(a_1 \cup b_2 \cup c_3 \cup d_4)(s) = 1$, i.e., $(a, b, c, d) \in W$.
$\diamond$
\vskip5pt

It now follows from the above claim that
\begin{align*}
\mathsf{F}_{s_1} &= \{ h_1 \cup h_2 \mid h_1 \in A , \, h_2 \in B \}, &
\mathsf{F}_{s_2} &= \{ h_3 \cup h_4 \mid h_3 \in C , \, h_4 \in D \}, \\
\mathsf{F}_{t_1} &= \{ h_1 \cup h_3 \mid h_1 \in A , \, h_3 \in C \}, &
\mathsf{F}_{t_2} &= \{ h_2 \cup h_4 \mid h_2 \in B , \, h_4 \in D \}.
\end{align*}
Let us focus on $s_1$.
Since $Y_1$ and $Y_2$ are nonempty, $\varcount{s_1} \geq 2$, so $s_1 = (s_{11} s_{12})$.
For $i, j \in \nset{2}$, let $Z_{ij} := \var(s_{1i}) \cap Y_j$.
Note that $Z_{11}$ and $Z_{12}$ are not both empty, and $Z_{21}$ and $Z_{22}$ are not both empty; without loss of generality, assume that $Z_{11}$ and $Z_{22}$ are nonempty.
Let
\begin{align*}
\mathsf{T} &:= \{ (a, b, c, d) \in \{0,1\}^{Z_{11}} \times \{0,1\}^{Z_{12}} \times \{0,1\}^{Z_{21}} \times \{0,1\}^{Z_{22}} \mid a \cup b \cup c \cup d \in \mathsf{T}_{s_1} \}, \\
\mathsf{F} &:= \{ (a, b, c, d) \in \{0,1\}^{Z_{11}} \times \{0,1\}^{Z_{12}} \times \{0,1\}^{Z_{21}} \times \{0,1\}^{Z_{22}} \mid a \cup b \cup c \cup d \in \mathsf{F}_{s_1} \}.
\end{align*}
We clearly have $\mathsf{T} = \overline{\mathsf{F}} := (\{0,1\}^{Z_{11}} \times \{0,1\}^{Z_{12}} \times \{0,1\}^{Z_{21}} \times \{0,1\}^{Z_{22}}) \setminus \mathsf{F}$.
Moreover,
\begin{align*}
\mathsf{T} &=
\{ (a, b, c, d) \in \{0,1\}^{Z_{11}} \times \{0,1\}^{Z_{12}} \times \{0,1\}^{Z_{21}} \times \{0,1\}^{Z_{22}} \mid a \cup b \in \mathsf{F}_{s_{11}}, \, c \cup d \in \mathsf{F}_{s_{12}} \}, \\
\mathsf{F} &=
\{ (a, b, c, d) \in \{0,1\}^{Z_{11}} \times \{0,1\}^{Z_{12}} \times \{0,1\}^{Z_{21}} \times \{0,1\}^{Z_{22}} \mid a \cup c \in A , \, b \cup d \in B \}.
\end{align*}
It now follows from Lemma~\ref{lem:aux-prod} (statement \ref{lem:aux-prod:2} if $Z_{12} = Z_{21} = \emptyset$; statement \ref{lem:aux-prod:3} if one of $Z_{12}$ and $Z_{21}$ is empty and the other is nonempty; statement \ref{lem:aux-prod:4} if both $Z_{12}$ and $Z_{21}$ are nonempty) that there is a variable $x_j \in \var(s_1)$ on which $s_1^\mathbf{G}$ does not depend.
This is in direct contradiction to Lemma~\ref{lem:surjective}, which asserts that $s_1^\mathbf{G}$ depends on every variable in $\var(s_1)$.
Therefore the case when $\{\var(s_1), \var(s_2)\} \neq \{\var(t_1), \var(t_2)\}$ does not occur.
\end{proof}

\section{Depth equivalence relations}
\label{sec:DepthEq}

In this section we study binary operations $*$ satisfying the property that two full linear terms agree on $*$ if and only if their corresponding binary trees are equivalent with respect to certain attributes related to the depths of the leaves.
More precisely, two binary trees are considered equivalent if their (right) depth sequences are congruent modulo some positive integer $k$.
Here all binary trees are ordered and their leaves are unlabeled unless otherwise stated.
 
\subsection{The $k$-right-depth-equivalence}\label{sec:RightDepth}

\begin{definition}
A groupoid $\mathbf{G} = (G,{*})$ and the corresponding binary operation $*$ are said to be \emph{right $k$\hyp{}associative} if
$\mathbf{G}$ satisfies the identity
\[
( [ x_1 x_2 \cdots x_{k+1} ]_\mathrm{R} x_{k+2} ) \approx ( x_1 [ x_2 \cdots x_{k+2} ]_\mathrm{R} ),
\]
where $[ \cdots ]_\mathrm{R}$ is a shorthand for the rightmost bracketing of the variables occurring between the square brackets, e.g., $[ x_1 x_2 \cdots x_{k+1} ]_\mathrm{R} = ( x_1 ( x_2 ( \cdots ( x_k x_{k+1} ) \cdots ) ) )$.
One can also define the \emph{left $k$\hyp{}associativity} similarly.
The left or right $k$\hyp{}associativity becomes the usual associativity when $k = 1$.
\end{definition}

\begin{example}
Typical examples of $k$-associative operations are the ones defined by $a * b := a + \omega b$ for all $a,b\in\CC$, where $\omega = e^{2 \pi i / k}$ is a $k$\hyp{}th primitive root of unity. 
This reduces to addition and subtraction when $k = 1, 2$, respectively.
\end{example}

Previous work~\cite{CatMod} showed that the equivalence relation on binary trees induced by the left $k$\hyp{}associativity is the same as the congruence relation on the left depth sequences of binary trees modulo $k$; this is called the \emph{$k$\hyp{}left\hyp{}depth\hyp{}equivalence relation} by the second author.
The number of equivalence classes is called the \emph{$k$-modular Catalan number}, which counts many restricted families of Catalan objects and has interesting closed formulas~\cite{CatMod}.
Of course, the right $k$\hyp{}associativity corresponds to the \emph{$k$\hyp{}right\hyp{}depth\hyp{}equivalence relation}, whose equivalence classes are also counted by the $k$\hyp{}modular Catalan number.

\begin{definition}
The $k$\hyp{}right\hyp{}depth\hyp{}equivalence relation extends immediately from binary trees with unlabeled leaves to ones with labeled leaves.
Let $T$ and $T'$ be binary trees with $n$ leaves labeled by $x_1, \dots, x_n$ (in an arbitrary order).
We say that $T$ and $T'$ are \emph{$k$\hyp{}right\hyp{}depth\hyp{}equivalent} if $\rho_T(x_i) \equiv \rho_{T'}(x_i) \pmod{k}$ for all $i \in \nset{n}$, i.e., the right depth sequences $\rho_T$ and $\rho_{T'}$ (see Subsection~\ref{sec:trees} for the definition of the right depth sequence) are componentwise congruent modulo $k$.
\end{definition}

Now we consider a stronger form of the right $k$\hyp{}associativity.
Suppose that a binary operation $*$ satisfies the property that any two full linear terms agree on $*$ if and only if $T_s$ and $T_t$ are $k$\hyp{}right\hyp{}depth\hyp{}equivalent, i.e., 
\begin{equation}\label{eq:k-depth}
\forall s, t\in F_n,\ s^{*} = t^{*} \Longleftrightarrow \rho_{T_s}(x_i) \equiv \rho_{T_t}(x_i) \pmod{k}, \quad i = 1, 2,\ldots, n. 
\end{equation}
It is clear that such a binary operation $*$ must be $k$\hyp{}right\hyp{}associative.
The above example $a*b := a + e^{2 \pi i / k} b$ satisfies property~\eqref{eq:k-depth} and another example is given by $f*g := xf + yg$ for all $x, y \in \CC[x,y] / (y^k - 1)$~\cite{VarCat}.
The associative spectrum of these examples is given by the $k$\hyp{}modular Catalan numbers mentioned above.
Our goal here is to determine the ac\hyp{}spectrum $\spac{*}$ of a binary operation $*$ satisfying property~\eqref{eq:k-depth} and its exponential generating function.
If $k = 1$ then we clearly have $\spac{*} = 1$ for all $n \geq 1$ with exponential generating function $\sum_{n \geq 1} t^n / n! = e^t - 1$.
Thus we assume $k \geq 2$ in the remainder of this section.

\begin{example}
Let us consider the subtraction operation $-$ on $\CC$.
For $n \geq 2$, one can check that the term operations in $F_n^{-}$ are precisely the operations of the form $(a_1, \dots, a_n) \mapsto \pm a_1 \pm a_2 \cdots \pm a_n$ with at least one plus sign and at least one minus sign.
Hence $\spac{-} = 2^n - 2$ and
\[
\begin{split}
\sum_{n = 1}^\infty \frac{\spac{-}}{n!} t^n 
& = t + \sum_{n \geq 2} \frac{2^n - 2}{n!} t^n \\
& = t + e^{2t} - 1 - 2t - 2(e^t-1-t) \\ 
& = t + e^{2t} - 2e^t + 1 \\
& = t + (e^t-1)^2.
\end{split}
\]
\end{example}

\begin{definition}
Let $T$ be any binary tree with $n \ge 2$ leaves.
For $i \in \nset{k}$, let $n_T(i)$ be the number of leaves in $T$ whose right depth is congruent to $i$ modulo $k$.
We say a sequence $(n_1, \ldots, n_k)$ of nonnegative integers is \emph{\textup{(}right\textup{)} $(n,k)$\hyp{}admissible} if it satisfies 
\begin{enumerate}[label={\upshape(\roman*)}]
\item
$n_1 \geq 1$, $n_k \geq 1$, and $n_1 + n_2 + \dots + n_k = n \geq 2$,
\item
if $n_i = 0$ for some $i \in \{2, 3, \dots, k-2\}$ then $n_{i+1} = 0$, and
\item
if $n_{k-1} = 0$ then $n_k = 1$.
\end{enumerate}
\end{definition}

\begin{lemma}
\label{lem:nk-adm-1}
Let $T$ be any binary tree with $n \geq 2$ leaves. 
Then $(n_T(1), \dots, n_T(k))$ is $(n,k)$\hyp{}admissible.
\end{lemma}

\begin{proof}
(i) We have $n_T(1) \geq 1$ and $n_T(k) \geq 1$ since the two leftmost leaves in $T$ have right depth $0$ and $1$, respectively. 
We have $n_T(1) + n_T(2) + \dots + n_T(k) = n$ since $T$ has $n$ leaves.

(ii) Let $i \in \{2, 3, \dots, k-2\}$.
Suppose there exists a leaf $v$ in $T$ whose right depth is congruent to $i+1$ modulo $k$.
If $v$ is a right child then the leftmost leaf of the subtree rooted at the left sibling $u$ of $v$ has right depth congruent to $i$ modulo $k$.
If $v$ is a left child then it must have an ancestor $u$ whose right depth is one less than that of $v$ (otherwise the right depth of $v$ would be zero), and the leftmost leaf in the subtree rooted at $u$ has right depth congruent to $i$ modulo $k$.
Therefore $n_T(i+1) > 0$ implies $n_T(i) > 0$, or in other words, $n_T(i) = 0$ implies $n_T(i+1) = 0$.

(iii) Similarly to (2), one can show that if a leaf of $T$ is different from the leftmost leaf but has right depth congruent to $k$, then there exists a leaf with right depth congruent to $k-1$ modulo $k$.
Hence $n_T(k-1) = 0$ implies $n_T(k) = 1$. 
\end{proof}

\begin{lemma}
\label{lem:nk-adm-2}
Let $(n_1, \dots, n_k)$ be any $(n,k)$\hyp{}admissible sequence.
Then there exists a binary tree $T$ with $n$ leaves such that $n_T(i) = n_i$ for all $i \in \nset{k}$.
\end{lemma}

\begin{proof}
We prove this lemma by induction on $n$.
There is no $(n,k)$\hyp{}admissible sequence for $n = 1$, and the claim is trivial for $n = 2$.
Assume $n \geq 3$ and let $(n_1, \dots, n_k)$ be an $(n,k)$\hyp{}admissible sequence.

If $n_1 > 1$ then applying the induction hypothesis to the $(n-1,k)$\hyp{}admissible sequence $(n_1 - 1, n_2, \dots, n_k)$ gives a binary tree $T'$ with $n-1$ leaves such that $n_{T'}(1) = n_1 - 1$ and $n_{T'}(i) = n_i$ for $i \in \{2, 3, \dots, k\}$, and the binary tree $T := T' \wedge \mathbf{1}$ satisfies $n_T(i) = n_i$ for all $i$, where $\mathbf{1}$ is the one-vertex tree.

Assume now $n_1 = 1$.
We must have $n_2 \geq 1$, since $n_2 = 0$ would imply $n_3 = \dots = n_{k-1} = 0$ and $n_k = 1$; hence $n = 2 < 3$, a contradiction.
Therefore we can apply the induction hypothesis to the $(n-1,k)$\hyp{}admissible sequence $(n_2, n_3, \dots, n_{k-1}, n_k-1, 1)$ and get a binary tree $T'$ with $n_{T'}(i-1) = n_i$ for $i \in \{2, 3, \dots, k-1\}$, $n_{T'}(k-1) = n_k - 1$, and $n_{T'}(k) = 1$. 
Then the binary tree $T := \mathbf{1} \wedge T'$ satisfies $n_T(1) = n_{T'}(k) = 1$, $n_T(i) = n_{T'}(i-1) = n_i$ for $i \in \{2, 3, \dots, k-1\}$, and $n_T(k) = n_{T'}(k-1) + 1 = n_k$.
\end{proof}

\begin{remark}
If $i \in \{1, k\}$ then let $m_T(i) := n_T(i) - 1$; otherwise let $m_T(i) := n_T(i)$.
It follows immediately from the previous two lemmas that
the sequences $(m_T(1), \dots, m_T(k))$ obtained from binary trees $T$ with $n$ leaves are precisely those sequences $(m_1, \dots, m_k)$ of nonnegative integers such that 
\begin{enumerate}[label={\upshape(\roman*)}]
\item $m_1 + m_2 + \dots + m_k = n - 2$, and
\item if $m_i = 0$ for some $i \in \{2, 3, \dots, k - 1\}$ then $m_{i+1} = 0$.
\end{enumerate}
By removing the trailing zeros, one sees that either $(m_1, \dots, m_k)$ corresponds to a composition (i.e., an ordered partition) of $n - 2$ with length at most $k$ if $m_1 > 0$, or $(m_2, \dots, m_k)$ corresponds to a composition of $n - 2$ with length at most $k - 1$ if $m_1 = 0$.
\end{remark}

Recall that the number of partitions of the set $\nset{n} = \{1, 2, \dots, n\}$ into $k$ (unordered) blocks is the \emph{Stirling number of the second kind} $S(n,k)$ 
whose exponential generating function is 
\[
\sum_{n \geq 1} S(n,k) \frac{t^n}{n!} = \frac{(e^t-1)^k}{k!}.
\]

\begin{theorem}\label{thm:k-associative}
Let $*$ be a binary operation satisfying property~\eqref{eq:k-depth} with $k \geq 2$.
Then
\[ 
\spac{*} = k! S(n,k) + n \sum_{0 \leq i \leq k-2} i! S(n-1,i), 
\qquad \forall n \geq 1,
\]
\[ \sum_{n=1}^\infty \frac{\spac{*}}{n!} t^n = (e^t - 1)^k + \sum_{0 \leq i \leq k-2} t(e^t-1)^{i}.
\]
\end{theorem}

\begin{proof}
We prove the desired formula for $\spac{*}$ by induction.
It is trivial when $n = 1$. Assume $n \geq 2$ below.

Property~\eqref{eq:k-depth} asserts that two terms $s, t \in F_n$ induce the same term operation on $*$ if and only if the right\hyp{}depth sequences $\rho_{T_s} := (\rho_{T_s}(x_1), \dots, \rho_{T_x}(x_n))$ and $\rho_{T_t} := (\rho_{T_t}(x_1), \dots, \rho_{T_t}(x_n))$ are componentwise congruent modulo $k$.
Therefore $\spac{*}$ equals the number of possible right\hyp{}depth sequences modulo $k$ of binary trees with $n$ leaves labeled with $x_1, \dots, x_n$.

By Lemmas \ref{lem:nk-adm-1} and \ref{lem:nk-adm-1}, a sequence $(d_1, \dots, d_n)$ with $1 \leq d_1, \dots, d_n \leq k$ is congruent modulo $k$ to the right depth sequence of a binary tree with $n$ leaves labeled by $x_1, \dots, x_n$ if and only if the sequence $(n_1, \dots, n_k)$ is $(n,k)$\hyp{}admissible, where 
\[
n_i := \#\{j \in \nset{n} : d_j = i\} \qquad \text{for $i = 1, \dots, k$.}
\]

First, assume that $n_1, \dots, n_k \geq 1$.
Then $(n_1, \dots, n_k)$ is $(n,k)$\hyp{}admissible if and only if $n_1 + \dots + n_k = n$.
Thus the sequences $(d_1, \dots, d_n)$ belonging to this case are in bijection with partitions of $\nset{n}$ into ordered blocks $\{ j \in \nset{n} : d_j = i\}$ for $i = 1, \dots, k$, and they are counted by $k! S(n,k)$.

Next, assume that there exists $i \in \nset{k-2}$ such that $n_{i+1} = 0$, where $i$ is as small as possible.
We have $n_{i+1} = \dots = n_{k-1} = 0$ and $n_k = 1$.
Thus the number of sequences $(d_1, d_2, \dots, d_n)$ belonging to this case is
\begin{align*}
n \sum_{1 \leq i \leq k-2} i! S(n-1,i)
\end{align*}

Combining the above two cases we have 
\begin{align*}
\spac{*}
&= k! S(n,k) + n \sum_{1 \leq i \leq k-2} i! S(n-1,i) \\
&= k! S(n,k) + n \sum_{0 \leq i \leq k-2} i! S(n-1,i).
\end{align*}
Consequently, we have the exponential generating function 
\begin{align*}
\sum_{n=1}^\infty \frac{\spac{*}}{n!} t^n
&= \sum_{n \geq 1} k! S(n,k) \frac{t^n}{n!} + \sum_{n \geq 1} \sum_{0 \leq i \leq k-2} i! S(n-1,i) \frac{t^n}{(n-1)!} \\
& = (e^t-1)^k + \sum_{0 \leq i \leq k-2} i! \sum_{n \geq 1} S(n-1,i) \frac{t^{n}}{(n-1)!} \\ 
& = (e^t-1)^k + \sum_{0 \leq i \leq k-2} t(e^t-1)^{i}.
\qedhere
\end{align*}
\end{proof}

\begin{remark}
When $k = 3$ we have $n \sum_{1 \leq i \leq k-2} i! S(n-1,i) = n$ for all $n \geq 2$.
When $k = 4$, the sequence $n \sum_{1 \leq i \leq k-2} i! S(n-1,i)$ is recorded in OEIS~\cite[A058877]{OEIS} and has the following simple closed formulas:
\[
n 2^{n-1} - n = \sum_{1 \leq j \leq n} (n - 2 + j) 2^{n-j-1} 
= \sum_{1 \leq j \leq n-1} \binom{n}{j} (n-j).
\]
\end{remark}

Recently, Hein and the first author~\cite{VarCat} generalized the $k$\hyp{}associativity to the \emph{$(k,\ell)$\hyp{}associativity}, based on the example $f * g := xf + yg$ for all $x, y \in \CC[x,y] / (x^k-1, y^\ell-1)$, which satisfies
\begin{equation}\label{eq:k-l-depth}
\forall s, t\in F_n,\ s^{*} = t^{*} \Longleftrightarrow \delta_{T_s}(x_i) \equiv \delta_{T_t}(x_i) \pmod{k},\quad \rho_{T_s}(x_i) \equiv \rho_{T_t}(x_i) \pmod{\ell}, \quad i = 1, 2,\ldots, n. 
\end{equation}
Computations give the following ac\hyp{}spectra for any operation $*$ satisfying the above property~\eqref{eq:k-l-depth}, which do not appear in OEIS.
\begin{itemize}
\item
$(k,\ell) = (2,2)$: $1$, $2$, $12$, $54$, $260$, $1080$, \ldots 

\item
$(k,\ell) = (3,2)$: $1$, $2$, $12$, $84$, $590$, $4110$, \ldots 

\item
$(k,\ell) = (4,2)$: $1$, $2$, $12$, $84$, $770$, $7080$, \ldots

\item
$(k,\ell) = (3,3)$: $1$, $2$, $12$, $108$, $960$, $9240$, \ldots
\end{itemize}

\subsection{The $k$-depth-equivalence}\label{sec:k-depth}
We now consider the operation on $\CC$ defined by $a * b := e^{2 \pi i / \ell} a + e^{2 \pi i / k} b$, which generalizes the example $a * b := a + e^{2 \pi i / k} b$ mentioned earlier.
When $k = \ell$, one sees that any two full linear terms agree on $*$ if and only if their corresponding binary trees are $k$\hyp{}depth\hyp{}equivalent, i.e.,
\begin{equation}\label{eq:depth} 
\forall s,t\in F_n,\ s^{*} = t^{*} \iff d_{T_s}(x_i) \equiv d_{T_t}(x_i) \pmod{k} .
\end{equation}
Further generalizations of depth equivalence were studied in recent work of the second author.

For $k = 2$, the resulting operation is
the \emph{double minus operation} $a \ominus b := -a - b$.
The first author, Mickey, and Xu~\cite{DoubleMinus} showed that if $\spa{\ominus}=1$ if $n=1$ and
\[
\spa{\ominus}
= \left\lfloor \frac{2^n}3 \right\rfloor = 
\begin{cases} 
\displaystyle \frac{2^{n}-1}{3}, & \text{if $n$ is even}, \\[10pt]
\displaystyle \frac{2^{n}-2}{3}, & \text{if $n$ is odd}
\end{cases}
\]
if $n\ge2$.
This coincides with an interesting sequence in OEIS~\cite[A000975]{OEIS} (except for the first term). 
Now we show that $\spac{\ominus}$ agrees with the well\hyp{}known \emph{Jacobsthal sequence}~\cite[A001045]{OEIS}.

\begin{theorem}\label{thm:ominus}
For $n \geq 1$ we have $\spac{\ominus} = (2^n - (-1)^n) / 3$.
\end{theorem}

\begin{proof}
Note first that every term operation in $B_n^{\ominus}$ is of the form $(a_1, \dots, a_n) \mapsto \pm a_1 \pm \dots \pm a_n$.
On the other hand, not every operation of this form is in $B_n^{\ominus}$.
Consider operations of this form with exactly $r$ plus signs in the defining expression; there are $\binom{n}{r}$ such operations.
It was shown in~\cite[Theorem~8]{DoubleMinus} that
\begin{enumerate}[label={\upshape(\roman*)}]
\item all operations of this form are in $B_n^{\ominus}$ if $n + r \equiv 2 \pmod{3}$ and $n \neq 2r - 1$,
\item all but the one with alternating signs are in $B_n^{\ominus}$ if $n + r \equiv 2 \pmod{3}$ and $n = 2r - 1$,
and
\item none of them is in $B_n^{\ominus}$ if $n + r \not\equiv 2 \pmod{3}$.
\end{enumerate}

Now, let us consider term operations in $F_n^{\ominus}$.
They are again of the form $(a_1, \dots, a_n) \mapsto \pm a_1 \pm \dots \pm a_n$.
We get no new operations in the first and third cases above, but one more in the second case: the operation with alternating signs is in $F_n^{\ominus}$.
Taking the sum over all possible values of $r$, we have $\spac{\ominus} = \spa{\ominus}$ if $n$ is even and $\spac{\ominus} = \spa{\ominus} + 1$ if $n$ is odd. 
This implies that $\spac{\ominus} = (2^n - (-1)^n) / 3$.
\end{proof}

One sees that this theorem holds for any binary operation $*$ satisfying property~\eqref{eq:depth} with $k=2$.
For $k \geq 3$, neither $\spa{*}$ nor $\spac{*}$ occurs in OEIS.
\begin{itemize}
\item
$\spa{*}$ for $k = 3$: $1$, $2$, $5$, $14$, $42$, $129$, $398$, $1223$, $3752$, $11510$, \ldots

\item
$\spa{*}$ for $k = 4$: $1$, $2$, $5$, $14$, $42$, $132$, $429$, $1429$, $4849$, $16689$, \ldots

\item
$\spac{*}$ for $k = 3$: $1$, $3$, $13$, $35$, $101$, $315$, \ldots

\item
$\spac{*}$ for $k = 4$: $1$, $3$, $13$, $75$, $285$, $1099$, \ldots
\end{itemize}

\section{Remarks and questions}
\label{sec:conclusion}

Cs\'{a}k\'{a}ny and Waldhauser~\cite[Section~4]{AssociativeSpectra1} examined the associative spectrum of every two\hyp{}element groupoid.
We obtain the ac-spectra of all two-element groupoids in Example~\ref{ex:two-element-groupoids} and Propositions~\ref{prop:implication} and \ref{prop:Sheffer}.

As possible directions for further research, one could study the associative spectra and the ac-spectra of groupoids satisfying some properties weaker than associativity.
We would like to mention a few examples of such properties that have emerged in different branches of algebra.

A groupoid $(G,*)$ is \emph{left alternative} (resp., \emph{right alternative}) if it satisfies the identity $(x*x)*y \approx x*(x*y)$ (resp., $y*(x*x)=(y*x)*x$).
A groupoid is \emph{alternative} if it is both left and right alternative.
An associative groupoid must be alternative, but not vice versa.

A Lie algebra is neither commutative nor associative, but it is \emph{flexible} in the sense that it satisfies the identity $x * (y * x) \approx (x * y) * x$.
Any commutative or associative operation must be flexible.
So flexibility becomes important when a binary operation is neither commutative nor associative, such as the multiplication of the sedenions, which is not even alternative (the same for all higher Cayley--Dickson algebras).
For example, the multiplication of the octonions is alternative but not associative.
In 1954, Richard Schafer~\cite{Schafer} examined the algebras generated by the Cayley--Dickson process over a field and showed that they satisfy the flexible identity.

A groupoid $(G,*)$ is \emph{power associative} if the subgroupoid generated by any element is associative.
Power associativity and alternativity are unrelated properties in the sense that neither implies the other, as shown, e.g., by the examples provided by Holin \cite[Appendix A.1.1]{Holin}.

Recall that the \emph{Jordan algebra} of $n \times n$ self\hyp{}adjoint matrices over $\RR$, $\CC$, or $\mathbb{H}$ with a product defined by $x \circ y := (xy+yx)/2$ is commutative and thus flexible. 
Although it is not associative, it is power associative and satisfies the \emph{Jordan identity} 
\[
(x \circ y) \circ (x \circ x) \approx x \circ (y \circ (x \circ x)).
\]
We showed in Section~\ref{sec:commutative-neutral} that the associative spectrum $\spa{\circ}$ of the Jordan algebra achieves the upper bound $C_{n-1}$ for an arbitrary binary operation and its ac\hyp{}spectrum $\spac{\circ}$ reaches the upper bound $D_{n-1}$ for commutative operations.

Another interesting example is the \emph{Okubo algebra}, which consists of all $3$-by-$3$ trace zero complex matrices with product defined as 
\[ x \circ y := axy + byx - \mathrm{tr}(xy) I_3 /3. \]
Here $I_3$ is the $3$-by-$3$ identity matrix and $a, b \in \CC$ satisfy $a+b = 3ab = 1$.
The Okubo algebra is flexible and power associative but not associative nor alternative.


\begin{thebibliography}{99}


\bibitem{Bergman}
C. Bergman, {\it Universal algebra}, Pure and Applied Mathematics (Boca Raton), 301, CRC, Boca Raton, FL, 2012.

\bibitem{Birkhoff}
G. Birkhoff, On the structure of abstract algebras, Proc.\ Cambr.\ Philos.\ Soc.\ {\bf 31} (1935), 433--454.

\bibitem{Subassociative}
M. S. Braitt and\ D. Silberger, Subassociative groupoids, Quasigroups Related Systems {\bf 14} (2006), no.~1, 11--26.

\bibitem{AssociativeSpectra1}
B. Cs\'{a}k\'{a}ny and\ T. Waldhauser, Associative spectra of binary operations, Mult.-Valued Log. {\bf 5} (2000), no.~3, 175--200.

\bibitem{EvenLempel}
S. Even and\ A. Lempel, Generation and enumeration of all solutions of the characteristic sum condition, Information and Control {\bf 21} (1972), 476--482.

\bibitem{GM}
A. Giorgilli and G. Molteni, Representation of a 2-power as sum of $k$ 2-powers: a recursive formula, J. Number Theory {\bf 133} (2013), no.~4, 1251--1261.

\bibitem{CatMod}
N. Hein and J. Huang, Modular Catalan numbers, European J. Comb.\ 61 (2017), 197--218.

\bibitem{VarCat}
N. Hein\ and\ J. Huang, Variations of the Catalan numbers from some nonassociative binary operations, Discrete Math.\ {\bf 345} (2022), no.~3, Paper No. 112711, 18 pp.

\bibitem{Holin}
H. Holin, Cayley integers (long version), arXiv: math/0506349.

\bibitem{Hamming}
J. Huang, Norton algebras of the Hamming graphs via linear characters, Electron.\ J. Combin.\ {\bf 28} (2021), no.~2, Paper No. 2.30, 36 pp.

\bibitem{DoubleMinus}
J. Huang, M. Mickey and\ J. Xu, The nonassociativity of the double minus operation, J. Integer Seq.\ {\bf 20} (2017), no.~10, Art. 17.10.3, 11 pp. 

\bibitem{Knuth}
D. E. Knuth, Letter to R. E. Tarjan and N. J. A. Sloane, Jul. 1975, \url{https://oeis.org/A007178/a007178.pdf}.

\bibitem{KrennWagner}
D. Krenn and S. Wagner, Compositions into powers of $b$: asymptotic enumeration and parameters, Algorithmica {\bf 75} (2016), no.~4, 606--631.

\bibitem{LST}
S. Lehr, J. Shallit and J. Tromp, On the vector space of the automatic reals, Theoret.\ Comput.\ Sci.\ {\bf 163} (1996), no.~1--2, 193--210. 

\bibitem{GraphAlgebra1}
E. Lehtonen and T. Waldhauser, Associative spectra of graph algebras I. Foundations, undirected graphs, antiassociative graphs, J. Algebraic Combin.\ {\bf 53} (2021), no.~3, 613--638.

\bibitem{GraphAlgebra2}
E. Lehtonen\ and\ T. Waldhauser, Associative spectra of graph algebras II. Satisfaction of bracketing identities, spectrum dichotomy, J. Algebraic Combin. {\bf 55} (2022), no.~2, 533--557. 

\bibitem{AssociativeSpectra2}
S. Liebscher and T. Waldhauser, On associative spectra of operations, Acta Sci.\ Math.\ (Szeged) {\bf 75} (2009), no.~3--4, 433--456.

\bibitem{Operads}
J.-L. Loday and B. Vallette, {\it Algebraic operads}, Grundlehren der mathematischen Wissenschaften, 346, Springer, Heidelberg, 2012.

\bibitem{OEIS}
OEIS Foundation Inc. (2011), The On-Line Encyclopedia of Integer Sequences, \url{http://oeis.org}.

\bibitem{Schafer}
R. D. Schafer, On the algebras formed by the Cayley-Dickson process, Amer. J. Math. {\bf 76} (1954), 435--446. 

\bibitem{EC2}
R. P. Stanley, {\it Enumerative combinatorics. Vol. 2}, Cambridge Studies in Advanced Mathematics, 62, Cambridge University Press, Cambridge, 1999.

\bibitem{Takacs}
L. Tak\'acs, Enumeration of rooted trees and forests, Math.\ Sci.\ {\bf 18} (1993) 1--10.

\end{thebibliography}
\end{document}